\documentclass[reqno]{amsart}

\usepackage{amssymb}
\usepackage{amsfonts}
\usepackage{amsmath}
\usepackage{color}
\usepackage{amsaddr}

\newtheorem{theorem}{Theorem}[section]

\newtheorem{corollary}[theorem]{Corollary}

\newtheorem{definition}[theorem]{Definition}

\newtheorem{lemma}[theorem]{Lemma}

\newtheorem{proposition}[theorem]{Proposition}

\theoremstyle{remark}

\newtheorem{remark}[theorem]{Remark}

\numberwithin{equation}{section}

\newcommand{\ep}{\varepsilon}

\DeclareMathOperator*{\essinf}{ess\,inf}

\begin{document}

\title[Relaxation of Nonlocal Cahn-Hilliard equations]{Well-posedness and global attractors for a non-isothermal viscous relaxation of nonlocal Cahn-Hilliard equations}

\author[Joseph L. Shomberg]{Joseph L. Shomberg}



\address{Department of Mathematics and Computer Science, Providence
College, Providence, RI 02918, USA, \\
\tt{{jshomber@providence.edu} }}

\date{\today}

\begin{abstract}
We investigate a non-isothermal viscous relaxation of some nonlocal Cahn-Hilliard equations.
This perturbation problem generates a family of solution operators, exhibiting dissipation and conservation. 
The solution operators admit a family of compact global attractors that are bounded in a more regular phase-space.
\end{abstract}

\maketitle 

\tableofcontents

\section{Introduction}

Inside a bounded domain (container) $\Omega\subset\mathbb{R}^3,$ we consider a phase separation model for a binary solution (e.g. a cooling alloy),
\[
\phi_t = \nabla\cdot[\kappa(\phi)\nabla\mu],
\]
where $\phi$ is the {\em{order-parameter}} (the relative difference of the two phases), $\kappa$ is the {\em{mobility function}} (which we set $\kappa\equiv1$ throughout this article), and $\mu$ is the {\em{chemical potential}} (the first variation of the free-energy $E$ with respect to $\phi$).
In the classical model,
\[
\mu = -\Delta\phi + F'(\phi) \quad \text{and} \quad E(\phi) = \int_\Omega \left( \frac{1}{2}|\nabla\phi|^2 + F(\phi) \right) dx,
\]
where $F$ describes the density of potential energy in $\Omega$ (e.g. the double-well potential $F(s)=(1-s^2)^2$).

Recently the nonlocal free-energy functional appears in the literature \cite{Giacomin-Lebowitz-97},
\[
E(\phi) = \int_\Omega\int_\Omega \frac{1}{4}J(x-y)(\phi(x)-\phi(y))^2 dxdy + \int_\Omega F(\phi) dx,
\]
hence, the {\em{chemical potential}} is, $\mu = a\phi - J*\phi + F'(\phi),$ where
\begin{align}
a(x) = \int_\Omega J(x-y) dy \quad \text{and} \quad (J*\phi)(x) = \int_\Omega J(x-y)\phi(y) dy.  \notag
\end{align}

In this article we consider the following problems: for $\alpha>0$, $\delta>0$, and $\ep>0$ the {\em{relaxation}} Problem P$_{\alpha,\ep}$ is, given $T>0$ and $(\phi_0,\theta_0)^{tr},$ find $(\phi^+,\theta^+)^{tr}$ satisfying
\begin{eqnarray}
\phi^{+}_t = \Delta\mu^{+} &\text{in}& \Omega\times(0,T)  \label{rel-1} \\ 
\mu^{+} = a\phi^{+} - J*\phi^{+} + F'(\phi^{+}) + \alpha \phi^{+}_t - \delta\theta^{+} &\text{in}& \Omega\times(0,T)  \label{rel-2} \\ 
\ep\theta_t^{+} - \Delta\theta^{+} = -\delta \phi_t^{+} &\text{in}& \Omega\times(0,T)  \label{rel-3} \\ 
\partial_n\mu^{+} = 0 &\text{on}& \Gamma\times(0,T)  \label{rel-4} \\ 
\partial_n\theta^{+} = 0 &\text{on}& \Gamma\times(0,T)  \label{rel-5} \\ 
\phi^{+}(x,0) = \phi_0(x) &\text{at}& \Omega\times\{0\}  \label{rel-6} \\ 
\theta^{+}(x,0) = \theta_0(x) &\text{at}& \Omega\times\{0\}. \label{rel-7}  
\end{eqnarray}

The main focus of this article is to examine the the asymptotic behavior of solutions to Problem P$_{\alpha,\ep}$, via global attractors, and the regularity of these attractors.
For ease of presentation, throughout we assume there is $\delta_0>0$ so that $\delta\in(0,\delta_0]$, and also $(\alpha,\ep)\in(0,1]\times(0,1].$

Let us now give some preliminary words on the motivation for using nonlocal diffusion.
First, in \cite[Equation (0.2)]{AVMRTM10} the nonlocal diffusion terms $a\phi-J*\phi$ appear as,
\[
\int_\Omega J(x-y)\left( \phi(x,t)-\phi(y,t) \right)dy,
\]
i.e. $a(x)=J*1.$
Heuristically, this integral term ``takes into account the individuals arriving at or leaving position $x$ from other places.''
In this setting, the term $a(x)\ge0$ is a factor of how many individuals arrive at position $x$.
Since the integration only takes place over $\Omega,$ individuals are not entering nor exiting the domain.
Hence, this representation is faithful to the desired mass conservation law we typically associate with Neumann boundary conditions. 
Although Neumann boundary conditions for the chemical potential $\mu$ make sense from the physical point of view of mass conservation, it is not necessarily true that the interface between the two phases is always orthogonal to the boundary, which is implied by the boundary condition $\partial_{n}\phi=0$ which commonly appears in the literature. 
This is partially alleviated by using nonlocal diffusion on $\phi.$

We report an important observation (cf. \cite[Equations (2.2)-(2.3)]{Gal&Miranville09}).

\begin{remark}
Once we have determined the values of $\phi(t)$ and $\theta(t)$ for any value $t=t^*$, then the value of the chemical potential $\mu$ can be found by solving the boundary value problem (e.g. we illustrate with Problem P$_{\alpha,\ep}$), 
\begin{eqnarray*}
\mu(t^*)-\alpha\Delta\mu(t^*) = a\phi(t^*) - J*\phi(t^*) + F'(\phi(t^*)) - \delta\theta(t^*) &\text{in}& \Omega, \\
\partial_n\mu(t^*)=0 &\text{on}& \Gamma. 
\end{eqnarray*}
\end{remark}

There is obvious motivation already in the literature to investigate Problem P$_{\alpha,\ep}$ from the point of view of a singular limit of a Caginalp type phase-field system (cf. \cite[Equations (1.1)-(1.3)]{Gal&Grasselli08}, \cite[Equations (1.1)-(1.3)]{GGM08-2} and \cite{Miranville&Zelik02}).
Of the non-isothermal, nonlocal Allen-Cahn system, 
\begin{equation}  \label{mot-1}
\left\{ \begin{array}{l} \alpha \phi_t + a\phi-J*\phi + F'(\phi) = \delta\theta \\ 
\ep_1\theta_t - \Delta\theta = -\delta\phi_t, \end{array} \right. 
\end{equation}
with $\alpha>0$, $\delta>0$, and $\ep_1>0,$ the singular limit $\ep_1\rightarrow 0^+$ formally recovers the following isothermal, viscous, nonlocal Cahn-Hilliard equation,\begin{equation}  \label{mot-1.1}
\phi_t - \Delta(a\phi-J*\phi+F'(\phi)+\alpha\phi_t) = 0. 
\end{equation}
Equation \eqref{mot-1.1} in the case where $F$ is a singular (logarithmic) potential was studied in \cite{Gal&Grasselli14}. 
We should also notice that when we iterate this procedure to an appropriate non-isothermal version of \eqref{mot-1.1}, the resulting system is equivalent to \eqref{mot-1.1}. 
Indeed, when we consider the system,
\begin{equation*}
\left\{ \begin{array}{l} \phi_t = \Delta\mu \\ 
\mu = a\phi - J*\phi + F'(\phi) + \alpha\phi_t - \delta\theta \\ 
\ep_2\theta_t -\Delta\theta = -\delta\phi_t, \end{array} \right. 
\end{equation*}
the formal limit $\ep_2\rightarrow0^+$ yields the isothermal, viscous, nonlocal Cahn-Hilliard equation,
\[
\varphi_t=\Delta(a\phi-J*\phi+F'(\phi)+\beta\varphi_t),
\]
where 
\[
\beta=\frac{\alpha}{1+\delta^2} \quad \text{and} \quad \varphi(t)=\phi((1+\delta^2)t).
\]
Moreover, these relations effectively mean speeding up time by a factor of $1+\delta^2$ is equivalent to `loosening' the viscosity by the same factor (in the sense that the strong damping term has a weaker affect).

Finally, we now mention that (cf. \cite{Krejci_Sprekels04}) the term $-\delta\phi_t$ could be thought of as the linearization $\frac{d}{dt}G(\phi)$ for some appropriate function $G$. 
In this case the internal energy is nonlinear in the order parameter $\theta$; i.e., $e:=\theta+G(\phi).$

The first goal of this article concerns determining the global well-posedness of the model problem Problem P$_{\alpha,\ep}$. 
Second, we wish to determine the asymptotic behavior of the solutions to Problem P$_{\alpha,\ep}$ up to the existence of global attractors (or universal attractors) for appropriate $\alpha$ and $\ep$.

The main points of this article are as follows:

\begin{itemize}

\item For Problem P$_{\alpha,\ep}$ we establish (global) well-posedness of weak solutions using minimal assumptions on the nonlinear term $F$.

\item The weak solutions generate a strongly continuous one-parameter family of solution operators; i.e., a semigroup, which in turn admits a bounded absorbing set and certain compactness properties. 
Consequently the associated dynamical system is gradient.

\item The semigroup also admits a global attractor.
We show the global attractor is bounded in a more regular space with $\sqrt{\alpha}\mu\in L^\infty(0,\infty;H^2(\Omega))$.
Each of these properties hold for every $\alpha\in(0,1]$ and $\ep\in(0,1]$.

\end{itemize}

The next section provides the functional framework behind Problem P$_{\alpha,\ep}$. 

\section{Preliminaries}

Now we detail some preliminaries that will be applied to both problems.
To begin, define the spaces $H:=L^2(\Omega)$ and $V:=H^1(\Omega)$ with norms denoted by, $\|\cdot\|$ and $\|\cdot\|_V$, respectively. 
Otherwise, we write the norm of the Banach space $X$ with $\|\cdot\|_X$.
The inner-product in $H$ is denoted by $(\cdot,\cdot)$.
Denote the dual space of $V$ by $V'$, and the dual paring in $V'\times V$ is denoted by $\langle\cdot,\cdot\rangle.$
For every $\psi\in V'$, we denote by $\langle \psi \rangle$ the average of $\psi$ over $\Omega$, that is, 
\[
\langle \psi \rangle := \frac{1}{|\Omega|}\langle\psi,1\rangle,
\]
where $|\Omega|$ is the Lebesgue measure of $\Omega.$
Throughout, we denote by $\hat\psi:=\psi-\langle\psi\rangle$ and for future reference, observe $\langle\hat\psi\rangle=\langle \psi-\langle\psi\rangle \rangle=0.$
We will refer to the following norm in $V'$, which is equivalent to the usual one,
\begin{align}
\|\psi\|^2_{V'} = \left\| A^{-1/2}_N(\psi - \langle \psi \rangle) \right\|^2 + \langle \psi \rangle^2.  \notag
\end{align}

Define the space $L^2_0(\Omega):=\{\phi\in L^2(\Omega):\langle \phi \rangle=0\}.$
Let $A_N=-\Delta:L^2_0(\Omega)\rightarrow L^2_0(\Omega)$ with domain $D(A_N)=\{\psi\in H^2(\Omega):\partial_n\psi=0 \ \mathrm{on} \ \Gamma\}$ denote the ``Neumann-Laplace'' operator.
Of course the operator $A_N$ generates a bounded analytic semigroup, denoted $e^{-A_Nt}$, and the operator is nonnegative and self-adjoint on $L^2(\Omega).$ 
Recall, the domain $D(A_N)$ is dense in $H^2(\Omega).$
Further, define $V_0:=\{\psi\in V:\langle \psi \rangle=0\}$, and $V_0':=\{\psi\in V':\langle \psi \rangle=0\}$.
Then $A_N:V\rightarrow V'$, $A_N\in\mathcal{L}(V,V')$, is defined by, for all $u,v\in V$,
\[
\langle A_N u,v \rangle = \int_\Omega \nabla u\cdot \nabla v dx.
\]
It is well known that the restriction $A_{N\mid V_0}$ maps $V_0$ to $V_0'$ isomorphically, and the inverse map $\mathcal{N}=A_N^{-1}:V_0'\rightarrow V_0,$ is defined by, for all $\psi\in V_0'$ and $f\in V_0$
\[
A_N\mathcal{N}\psi=\psi, \quad \mathcal{N}A_Nf=f.
\]
Additionally, these maps satisfy the relations, for all $u\in V_0$ and $v,w\in V_0',$
\begin{align}
\langle A_N u,\mathcal{N}v\rangle = \langle u,v \rangle,  \label{NLr-1} \\ 
\langle v,\mathcal{N}w \rangle = \langle w,\mathcal{N}v \rangle.  \notag
\end{align}
The Sobolev space $V$ is endowed with the norm,
\begin{equation}  \label{H1-norm}
\|\psi\|^2_{V}:= \|\nabla\psi\|^2 + \langle \psi \rangle^2.
\end{equation}
Denote by $\lambda_\Omega>0$ the constant in the Poincar\'{e}-Wirtinger inequality,
\begin{equation}  \label{Poincare}
\|\psi-\langle\psi\rangle\| \le \sqrt{\lambda_\Omega}\|\nabla\psi\|.
\end{equation}
Whence, for $c_\Omega:=\max\{\lambda_\Omega,1\}$, there holds, for all $\psi\in V,$
\begin{align}
\|\psi\|^2 & \le \lambda_\Omega\|\nabla\psi\|^2 + \langle\psi\rangle^2  \label{Poincare2} \\ 
& \le c_\Omega\|\psi\|^2_{V}.  \notag
\end{align}

For each $m\ge0$, $\alpha>0$, and $\ep>0$ define the following energy phase-space for Problem P$_{\alpha,\ep}$,
\begin{align}
\mathbb{H}^{\alpha,\ep}_m:=\{ \zeta=(\phi,\theta)^{tr} \in H\times H : |\langle \phi \rangle|, |\langle \theta \rangle| \le m \},  \notag 
\end{align}
which is Hilbert when endowed with the $\alpha,\ep$-dependent norm whose square is given by,
\begin{align}
\|\zeta\|^2_{\mathbb{H}^{\alpha,\ep}_m} & := \|\phi\|^2_{V'} + \alpha\|\phi\|^2 + \ep\|\theta\|^2.  \notag 
\end{align}
When we are concerned with the dynamical system associated with Problem P$_{\alpha,\ep}$, we will utilize the following metric space
\begin{align}
\mathcal{X}^{\alpha,\ep}_m := \left\{ \zeta=(\phi,\theta)^{tr}\in\mathbb{H}^{\alpha,\ep}_m : F(\phi)\in L^1(\Omega) \right\},  \notag
\end{align}
endowed with the metric
\begin{align}
d_{\mathcal{X}^{\alpha,\ep}_m}(\zeta_1,\zeta_2) := \|\zeta_1-\zeta_2\|_{\mathbb{H}^{\alpha,\ep}_m} + \left| \int_\Omega F(\phi_1)dx - \int_\Omega F(\phi_2)dx \right|^{1/2}.  \notag
\end{align}
We also define the more regular phase-space for Problem P$_{\alpha,\ep}$,
\begin{align}
\mathbb{V}^{\alpha,\ep}_m:=\{ \zeta=(\phi,\theta)^{tr} \in V\times V : |\langle \phi \rangle|, |\langle \theta \rangle| \le m \},  \notag 
\end{align}
with the norm whose square is given by, $\|\zeta\|^2_{\mathbb{V}^{\alpha,\ep}_m} := \|\phi\|^2 + \alpha\|\phi\|^2_V + \ep\|\theta\|^2_V.$

The following assumptions on $J$ and $F$ are based on \cite{Frigeri&Grasselli12,Gal&Grasselli14}:

\begin{description}

\item[(H1)] $J\in W^{1,1}(\mathbb{R}^3)$, $J(-x)=J(x)$, and $a(x):=\int_\Omega J(x-y) dy > 0$ a.e. in $\Omega$.

\item[(H2)] $F\in C^{2,1}_{loc}(\mathbb{R})$ and there exists $c_0>0$ such that, for all $s\in\mathbb{R},$
\begin{align}
F''(s) + \inf_{x\in\Omega}a(x) \ge c_0.  \notag
\end{align}

\item[(H3)] There exists $c_1>\frac{1}{2}\|J\|_{L^1(\mathbb{R}^3)}$ and $c_2\in\mathbb{R}$ such that, for all $s\in\mathbb{R},$ 
\begin{align}
F(s)\ge c_1s^2 - c_2.  \notag 
\end{align}

\item[(H4)] There exists $c_3>0$, $c_4\ge0,$ and $p\in(1,2]$ such that, for all $s\in\mathbb{R},$
\begin{align}
|F'(s)|^p \le c_3|F(s)| + c_4.  \notag 
\end{align}

\item[(H5)] There exist $c_5,c_6>0,$ and $q>0$ such that, for all $s\in\mathbb{R},$
\begin{align}
F''(s) + \inf_{x\in\Omega}a(x) \ge c_5|s|^{2q} - c_6.  \notag
\end{align}

\end{description}

Let us make some remarks and report some important consequences of these assumptions. 
From \cite[Remark 2]{CFG12}: assumption (H2) implies that the potential $F$ is a quadratic perturbation of a (strictly) convex function. 
Indeed, if we set $a^*:=\|a\|_{L^\infty(\Omega)}$, then $F$ can be represented as 
\begin{equation}  \label{convex}
F(s)=G(s)-\frac{a^*}{2}s^2,
\end{equation}
with $G\in C^2(\mathbb{R})$ being strictly convex, since $G''\ge c_0$.
With (H3), for each $m\ge0$ there are constants $c_7,c_{8},c_{9},c_{10}>0$ (with $c_{8}$ and $c_{9}$ depending on $m$ and $F$) such that,
\begin{align}
F(s)-c_7\le c_{8}(s-m)^2 + F'(s)(s-m),  \label{Fcons-1}
\end{align}
\begin{align}
\frac{1}{2}|F'(s)|(1+|s|) \le F'(s)(s-m) + c_{9},  \label{Fcons-2}
\end{align}
and
\begin{align}
|F(s)| - c_{10} \le |F'(s)|(1+|s|).  \label{Fcons-3}
\end{align}
The last inequality appears in \cite[page 8]{Gal&Miranville09}.
With the positivity condition (H3), it follows that, for all $s\in\mathbb{R},$
\begin{equation}  \label{Fcons-3.1}
|F'(s)| \le c_3|F(s)|+c_4.
\end{equation}

{\em{A word of notation:}} In many calculations, functional notation indicating dependence on the variable $t$ is dropped; for example, we will write $\psi$ in place of $\psi(t)$. 
Throughout the article, $C>0$ will denote a \emph{generic} constant, while $Q:\mathbb{R}_{+}^d\rightarrow \mathbb{R}_{+}$ will denote a \emph{generic} increasing function in each of the $d$ components. 
Unless explicitly stated, all of these generic terms will be independent of the parameters $\alpha,$ $\delta,$ $\ep$, $T,$ and $m$.
Finally, throughout we will use the following abbreviations
\begin{align}
c_J:=\|J\|_{L^1(\Omega)} \quad \text{and} \quad d_J:=\|\nabla J\|_{L^1(\Omega)}.  \label{not-1}
\end{align}

\section{The relaxation Problem P$_{\alpha,\ep}$}

\subsection{Global well-posedness of Problem P$_{\alpha,\ep}$}

\begin{definition}  \label{d:ws}
For $T>0$, $\delta_0>0$, $\delta\in(0,\delta_0]$, $(\alpha,\ep)\in(0,1]\times(0,1]$, and $\zeta_0=(\phi_0,\theta_0)^{tr}\in H\times H$ with $F(\phi_0)\in L^1(\Omega)$, we say that $\zeta=(\phi,\theta)^{tr}$ is a {\sc{weak solution}} of Problem P$_{\alpha,\ep}$ on $[0,T]$ if $\zeta=(\phi,\theta)^{tr}$ satisfies 
\begin{align}
& \phi\in C([0,T];H) \cap L^2(0,T;V),  \label{ws-1} \\
& \phi_t\in L^2(0,T;V'),  \label{ws-2} \\ 
& \sqrt{\alpha}\phi_t\in L^2(0,T;V),  \label{ws-2.1} \\ 
& \mu=a(x)\phi - J*\phi + F'(\phi) + \alpha\phi_t - \delta\theta \in L^2(0,T;V),  \label{ws-3} \\
& \theta\in C([0,T];H) \cap L^2(0,T;V),  \label{ws-4} \\
& \theta_t\in L^2(0,T;V').  \label{ws-5} 
\end{align}
In addition, upon setting, 
\begin{align}
\rho = \rho(x,\phi) := a(x)\phi + F'(\phi),  \label{ws-6}
\end{align}
for every $\varphi,\vartheta\in V,$ there holds, for almost all $t\in(0,T),$
\begin{align}
\langle\phi_t,\varphi\rangle + (\nabla\rho,\nabla\varphi) - (\nabla (J*\phi),\nabla\varphi) + \alpha(\nabla\phi_t,\nabla\varphi) & = \delta(\nabla\theta,\nabla\varphi)  \label{ws-7} \\ 
\ep\langle\theta_t,\vartheta\rangle + (\nabla \theta,\nabla\vartheta) & = -\delta\langle\phi_t,\vartheta\rangle.  \label{ws-8}
\end{align}
Also, there holds,
\begin{align}
\phi(0) = \phi_0 \quad \text{and} \quad \theta(0) & = \theta_0.  \label{ws-9}
\end{align}
We say that $\zeta=(\phi,\theta)^{tr}$ is a {\sc{global weak solution}} of Problem P$_{\alpha,\ep}$ if it is a weak solution on $[0,T]$, for any $T>0.$
The initial conditions \eqref{ws-9} hold in the $L^2$-sense; i.e., for every $\varphi,\vartheta\in V,$
\begin{align}
(\phi(0),\varphi) = (\phi_0,\varphi) \quad \text{and} \quad (\theta(0),\vartheta)=(\theta_0,\vartheta)  \label{L2-initial-theta}
\end{align}
hold.
\end{definition}

It is well-known that the average value of $\phi$ is conserved (cf. e.g. \cite[Section III.4.2]{Temam01}). 
Indeed, taking $\varphi=1$ in \eqref{ws-7} yields, $\frac{\partial}{\partial t}\int_\Omega \phi(x,t) dx = 0$ and we naturally recover the {\em{conservation of mass}}
\begin{align}
\langle \phi(t) \rangle & = \langle \phi_0 \rangle.  \label{con-mass}
\end{align}
In addition to \eqref{con-mass}, taking $\vartheta=1$ in \eqref{ws-8} yields $\frac{\partial}{\partial t}\int_\Omega \theta(x,t) dx = 0$ and we also establish 
\begin{align}
\langle \theta(t) \rangle & = \langle \theta_0 \rangle \quad \text{as well as} \quad \partial_t\langle \phi(t) \rangle = \partial_t\langle \theta(t) \rangle = 0.  \label{con-heat}
\end{align}
Together, \eqref{con-mass} and \eqref{con-heat} constitute {\em{conservation of enthalpy}}.

\begin{theorem}  \label{t:existence}
Assume (H1)-(H5) hold with $p\in(\frac{6}{5},2]$ and $q\ge\frac{1}{2}$. 
For any $\zeta_0=(\phi_0,\theta_0)^{tr}\in H\times H$ with $F(\phi_0)\in L^1(\Omega)$, there exists a global weak solution $\zeta=(\phi,\theta)^{tr}$ to Problem P$_{\alpha,\ep}$ in the sense of Definition \ref{d:ws} satisfying the additional regularity, for any $T>0$, 
\begin{eqnarray}
\phi & \in & L^\infty(0,T;L^{2+2q}(\Omega)),  \label{wk-0.001} \\
\sqrt{\alpha}\phi & \in & L^\infty(0,T;V),  \label{wk-0.01} \\
F(\phi) & \in & L^\infty(0,T;L^1(\Omega)),  \label{wk-0.2} \\ 
\theta_t & \in & L^2(0,T;H).  \label{wk-0.3}
\end{eqnarray}
Furthermore, setting 
\begin{align}
\mathcal{E}_\ep(t):=\frac{1}{4} \int_\Omega \int_\Omega J(x-y)\left( \phi(x,t)-\phi(y,t) \right)^2 dxdy + \int_\Omega F(\phi(x,t)) dx + \frac{\ep}{2}\int_\Omega \theta(t)^2 dx,  \label{wk-1}
\end{align} 
the following energy equality holds, for all $\zeta_0=(\phi_0,\theta_0)^{tr}\in \mathbb{H}^{\alpha,\ep}_m$ with $F(\phi_0)\in L^1(\Omega)$, and $t\in[0,T],$ 
\begin{align}
& \mathcal{E}_\ep(t) + \int_0^t \left( \|\nabla\mu(s)\|^2 + \alpha\|\phi_t(s)\|^2 + \|\nabla\theta(s)\|^2 \right) ds = \mathcal{E}_\ep(0).  \label{wk-2}
\end{align}
\end{theorem}

\begin{proof}
We follow the proofs of \cite[Theorem 1]{CFG12} and \cite[Theorem 2.1]{Porta-Grasselli-2014}.
The proof proceeds in several steps.
The existence proof begins with a Faedo-Galerkin approximation procedure to which we later pass to the limit. 
We first assume that $\phi_0\in D(A_N)$ and $\theta_0\in H.$ (The first assumption will be used to show that there is a sequence $\{\phi_{0n}\}_{n=1}^\infty$ such that $\phi_{0n}\rightarrow\phi_{0}$ in $H^2(\Omega)$ as well as $L^\infty(\Omega)$, which will be important in light of the fact that $F(\phi_{0n})$ is of arbitrary polynomial growth per assumptions (H1)-(H5).)
The existence of a weak solution for $\phi_0\in H$ with $F(\phi_0)\in L^1(\Omega)$ will follow from a density argument and by exploiting the fact that the potential $F$ is a quadratic perturbation of a convex function (cf. equation \eqref{convex}).

{\em{Step 1. (Construction and boundedness of approximate maximal solutions)}} Recall that the linear operator $A_N+I$ is positive and self-adjont on $H$. 
Then we have a complete system of eigenfunctions $\{\psi_i\}_{i=1}^\infty$ of the eigenvalue problem $(A_N+I)\psi_i=\lambda_i\psi_i$ in $H$ with $\psi_i\in D(A_N)=\{\chi\in H^2(\Omega):\partial_n\chi=0\ \text{on}\ \Gamma\}$.
We know by spectral theory that the eigenvalues may be ordered and counted according to their multiplicities in order to form a (real) diverging sequence. 
The set of respective eigenvectors, $\{\psi_i\}_{i=1}^\infty$, forms an orthogonal basis in $V$, which we may assume is orthonormal in $H$.

Define the subspaces 
\[
\Psi_n:={\textrm{span}}\{\psi_1,\psi_2,\dots,\psi_n\} \quad \text{and} \quad \Psi_\infty:=\bigcup_{n=1}^\infty \Psi_n.
\]
By construction, clearly $\Psi_\infty$ is dense in $D(A_N).$
Then, for any fixed $T>0$ and $n\in\mathbb{N},$ we will seek functions of the form 
\begin{equation}
\phi_n(t)=\sum_{k=1}^{n}a_k(t)\psi_k \quad \text{and} \quad \theta_n(t)=\sum_{k=1}^{n}b_k(t)\psi_k,  \label{wk-2.5}
\end{equation}
that solve the following approximating problems for any $\delta_0>0,$ $\delta\in(0,\delta_0],$ $(\alpha,\ep)\in(0,1]\times(0,1]$, and for all $t\in[0,T],$
\begin{align}
(\phi_n',\varphi) + (\nabla\rho_n(\cdot,\phi_n),\nabla\varphi) - (\nabla J*\phi_n,\nabla\varphi) + \alpha(\nabla\phi'_n,\nabla\varphi) & = \delta(\nabla\theta_n,\nabla\varphi),  \label{wk-3} \\ 
\ep(\theta_n',\vartheta) + (\nabla\theta_n,\nabla\vartheta) & = -\delta(\phi_n',\vartheta),  \label{wk-4} \\
\rho_n=\rho(\cdot,\phi_n) & := a(\cdot)\phi_n + F'(\phi_n),  \label{wk-5} \\ 
\mu_n = P_n(\rho_n - J*\phi_n & + \alpha\phi_n' - \delta\theta_n),  \label{wk-6} \\ 
(\phi_n(0),\varphi) & = (\phi_{0n},\varphi),  \label{wk-7} \\ 
(\theta_n(0),\vartheta) & = (\theta_{0n},\vartheta),  \label{wk-8}
\end{align}
for every $\varphi,\vartheta\in\Psi_n$, and where $\phi_{0n}=P_n\phi_0$ and $\theta_{0n}=P_n\theta_0$; $P_n$ being the $n$-dimensional projection of $H$ onto $\Psi_n$. 
Throughout the remainder of the proof we set $M_0:=\langle \phi_0 \rangle$ and $N_0:=\langle \theta_0 \rangle.$
The functions $a_i$ and $b_i$ are assumed to be (at least) $C^2((0,T))$.
It is also worth noting that \eqref{con-mass} and \eqref{con-heat}, also hold for the discretized functions $\phi_n$ and $\theta_n$.

To show the existence of at least one solution to (\ref{wk-3})-(\ref{wk-8}), we now suppose that $n$ is fixed and we take $\varphi=\phi_k$ and $\vartheta=\theta_k$ for some $1\le k\le n$. 
Then substituting the discretized functions (\ref{wk-2.5}) into (\ref{wk-3})-(\ref{wk-8}), we arrive at a system of $n$ ODEs in the unknowns $a_k=a_k(t)$ and $b_k=b_k(t)$ on $\Psi_n$. 
Since $J\in W^{1,1}(\mathbb{R}^3)$ and $F\in C^{2,1}_{loc}(\mathbb{R})$, we may apply Cauchy's/Carath\'{e}odory's theorem for ODEs to find that there is $T_n\in(0,T)$ such that $a_k,b_k\in C^2((0,T_n))$, for $1\le k\leq n$, and (\ref{wk-3})-(\ref{wk-4}) hold in the classical sense for all $t\in[0,T_n]$. 
Since $F'\in C^1(\mathbb{R})$, this argument shows the existence of a unique maximal solution to the projected problem \eqref{wk-3}-\eqref{wk-8}.

Now we need to derive some {\em{a priori}} estimates to apply to the approximate maximal solutions to show that $T_n=+\infty$, for every $n\ge1$, and that the corresponding sequences $\phi_n$, $\theta_n$ and $\mu_n$ are bounded in some appropriate function spaces. 
To begin, we take $\varphi=\mu_n$ as a test function in \eqref{wk-3} and $\vartheta=\theta_n$ as a test function in \eqref{wk-4}, to obtain
\begin{align}
(\phi_n',\mu_n) + (\nabla\rho(\cdot,\phi_n),\nabla\mu_n) - (\nabla J*\phi_n,\nabla\mu_n) + \alpha(\nabla\phi'_n,\nabla\mu_n) & = \delta(\nabla\theta_n,\nabla\mu_n),  \label{wk-9} 
\end{align}
and
\begin{align}
\frac{\ep}{2}\frac{d}{dt}\|\theta_n\|^2 + \|\nabla\theta_n\|^2 = -\delta(\phi'_n,\theta_n).  \label{wk-9.1} 
\end{align}
Now we write (recall $J$ is even by (H1), so, in $H$, $(J*\phi_n)\phi'_n=(J*\phi'_n)\phi_n$),
\begin{align}
(\phi'_n,\mu_n) & = (\phi'_n,a\phi_n - J*\phi_n + F'(\phi_n) + \alpha\phi_n' - \delta\theta_n)  \notag \\
& = \frac{d}{dt}\left\{ \frac{1}{2}\|\sqrt{a}\phi_n\|^2 - \frac{1}{2}(J*\phi_n,\phi_n) + \int_\Omega F(\phi_n)dx \right\} + \alpha\|\phi'_n\|^2 - \delta(\phi'_n,\theta_n)  \notag \\
& = \frac{d}{dt}\left\{ \frac{1}{4} \int_\Omega\int_\Omega J(x-y)\left(\phi_n(x)-\phi_n(y)\right)^2 dxdy + \int_\Omega F(\phi_n)dx \right\} + \alpha\|\phi'_n\|^2 - \delta(\phi'_n,\theta_n).  \label{wk-10}
\end{align}
Also, 
\begin{align}
(\nabla\rho(\cdot,\phi_n),\nabla\mu_n) & = -(\rho(\cdot,\phi_n),\Delta\mu_n)  \notag \\ 
& = (-\rho_n,\Delta\mu_n)  \notag \\ 
& = (\nabla\rho_n,\nabla\mu_n),  \notag
\end{align}
where $\rho_n:=P_n\rho(\cdot,\phi_n)=\mu_n+P_n(J*\phi_n)-\alpha\phi'_n+\delta\theta_n.$
Hence, 
\begin{align}
(\nabla\rho(\cdot,\phi_n),\nabla\mu_n) = \|\nabla\mu_n\|^2 + (\nabla(P_n(J*\phi_n)),\nabla\mu_n) - \alpha(\nabla\phi'_n,\nabla\mu_n) + \delta(\nabla\theta_n,\nabla\mu_n).  \label{wk-11}
\end{align}
Combining \eqref{wk-1} (with the discritized functions), \eqref{wk-9}-\eqref{wk-11} yields the differential identity, 
\begin{align}
& \frac{d}{dt}\mathcal{E}_\ep + \|\nabla\mu_n\|^2 + \|\nabla\theta_n\|^2 + \alpha\|\phi'_n\|^2 + (\nabla(P_n(J*\phi_n)),\nabla\mu_n) - (\nabla J*\phi_n,\nabla\mu_n) = 0.  \label{wk-12} 
\end{align}
Estimating the two products in \eqref{wk-12}, we find
\begin{align}
(\nabla(P_n(J*\phi_n)),\nabla\mu_n) & \le \|\nabla(P_n(J*\phi_n))\|^2 + \frac{1}{4}\|\nabla\mu_n\|^2  \notag \\ 
& \le \|(A_N+I)^{1/2}P_n(J*\phi_n)\|^2 + \frac{1}{4}\|\nabla\mu_n\|^2  \notag \\ 
& \le \|\nabla J*\phi_n\|^2 + \|J*\phi_n\|^2 + \frac{1}{4}\|\nabla\mu_n\|^2  \notag \\ 
& \le (d_J^2+c_J^2)\|\phi_n\|^2 + \frac{1}{4}\|\nabla\mu_n\|^2,  \label{wk-12.1} 
\end{align}
and
\begin{align}
(\nabla J*\phi_n,\nabla\mu_n) & \le d_J^2\|\phi_n\|^2 + \frac{1}{4}\|\nabla\mu_n\|^2.  \label{wk-12.2}
\end{align}
Observe that with the aid of hypothesis (H3), there holds 
\begin{align}
\mathcal{E}_\ep & = \frac{1}{4} \int_\Omega\int_\Omega J(x-y)\left(\phi_n(x)-\phi_n(y)\right)^2 dxdy + \int_\Omega F(\phi_n)dx + \frac{\ep}{2}\int_\Omega\theta^2_ndx  \notag \\ 
& = \frac{1}{2}\|\sqrt{a}\phi_n\|^2 - \frac{1}{2}(J*\phi_n,\phi_n) + \int_\Omega F(\phi_n)dx + \frac{\ep}{2}\|\theta_n\|^2  \notag \\ 
& \ge \frac{1}{2}\int_\Omega\left( a+2c_1-\|J\|_{L^1(\Omega)} \right)\phi_n^2 dx - c_2|\Omega| + \frac{\ep}{2}\|\theta_n\|^2  \notag \\ 
& \ge \left( c_1 - \frac{c_J}{2} \right)\|\phi_n\|^2 - c_2|\Omega| + \frac{\ep}{2}\|\theta_n\|^2.  \label{wk-13}
\end{align}
Now, combining \eqref{wk-12}-\eqref{wk-12.2} and integrating the resulting inequality with respect to $t$ over $(0,T_n)$ and applying \eqref{wk-13} to the result produces,
\begin{align}
\left( c_1 - \frac{c_J}{2} \right) & \|\phi_n(t)\|^2 + \frac{\ep}{2}\|\theta_n(t)\|^2 + \frac{1}{2}\int_0^t \|\nabla\mu_n(s)\|^2 ds + \int_0^t \|\nabla\theta_n(s)\|^2 ds + \alpha \int_0^t \|\phi'_n(s)\|^2 ds  \notag \\
& \le (c_J^2+2d_J^2)\int_0^t \|\phi_n(s)\|^2 ds + \mathcal{E}_\ep(0) + c_2|\Omega|.  \notag
\end{align}
Using the basic estimate $\|P_n\psi\| \le \|\psi\|$ we find,
\begin{align}
\mathcal{E}_\ep(0) & = \frac{1}{4} \int_\Omega\int_\Omega J(x-y)\left(\phi_{0n}(x)-\phi_{0n}(y)\right)^2 dxdy + \int_\Omega F(\phi_{0n})dx + \frac{\ep}{2}\int_\Omega\theta^2_{0n}dx  \notag \\ 
& \le C(c_J,|\Omega|)\|\phi_{0}\|^2 + \int_\Omega F(\phi_{0}) dx + \frac{1}{2}\|\theta_{0}\|^2.  \label{wk-13.9}
\end{align}
The hypothesis that $F(\phi_0)\in L^1(\Omega)$ where $\phi_0\in D(A_N)$ implies that $\phi_{0n}\rightarrow\phi_0$ in $H^2(\Omega)$, and hence $L^\infty(\Omega)$.
Moreover,
\begin{align}
\|\phi_n(t)\|^2 + \ep\|\theta_n(t)\|^2 & + \int_0^t \|\nabla\mu_n(s)\|^2 ds + \int_0^t \|\theta_n(s)\|^2_{V} ds + \alpha \int_0^t \|\phi'_n(s)\|^2 ds  \notag \\
& \le \frac{1}{\nu_0}(c_J^2+2d_J^2)\int_0^t \|\phi_n(s)\|^2 ds + Q(\|\zeta_0\|_{\mathbb{H}^{\alpha,\ep}_m}) + N_0^2\cdot T,  \label{wk-14}
\end{align}
where (and with (H3)),
\begin{equation}  \label{wk-14.1}
\nu_0=\nu_0(J):=\min\left\{c_1 - \frac{c_J}{2}, \frac{1}{2} \right\}>0,
\end{equation}
and where the extra term appearing on the right-hand side of \eqref{wk-14} is to make the $V$ norm for $\theta_n$ on the left-hand side.
Since the right-hand side of \eqref{wk-14} is independent of $n$ and $t$, we deduce, by means of a Gr\"onwall inequality, that $T_n=+\infty,$ for every $n\ge1,$ i.e., the projected problem \eqref{wk-3}-\eqref{wk-8} has a unique global in time solution as $T>0$ is arbitrary, and \eqref{wk-14} is satisfied for every $t\ge0.$
Furthermore, from \eqref{wk-14}, we obtain the following estimates for any given $0<T<+\infty,$
\begin{eqnarray}
\phi_n & ~\text{is uniformly bounded in}~ & L^\infty(0,T;H),  \label{wk-15} \\
\theta_n & ~\text{is uniformly bounded in}~ & L^\infty(0,T;H),  \label{wk-16} \\
\nabla\mu_n & ~\text{is uniformly bounded in}~ & L^2(0,T;H),  \label{wk-17} \\
\theta_n & ~\text{is uniformly bounded in}~ & L^2(0,T;V), \label{wk-18} \\ 
\sqrt{\alpha}\phi'_n & ~\text{is uniformly bounded in}~ & L^2(0,T;H),  \label{wk-19} \\  
F(\phi_n) & ~\text{is uniformly bounded in}~ & L^\infty(0,T;L^1(\Omega)).  \label{wk-20}
\end{eqnarray}
(The last inclusion follows from the definition of $\mathcal{E}_\ep$ and \eqref{wk-13}.)

Now to show 
\begin{eqnarray}
\phi_n & ~\text{is uniformly bounded in}~ & L^2(0,T;V),  \label{wk-20.01}
\end{eqnarray}
we observe the two basic estimates hold for every $\eta>0,$
\begin{align}
(\nabla\mu_n,\nabla\phi_n) \le \frac{1}{4c_0\eta}\|\nabla\mu_n\|^2 + c_0\eta\|\nabla\phi_n\|^2,  \notag
\end{align}
and 
\begin{align}
(\nabla\mu_n,\nabla\phi_n) & \ge c_0\|\nabla\phi_n\|^2 - \frac{d_J^2}{2c_0\eta}\|\phi_n\|^2 - 2c_0\eta\|\nabla\phi_n\|^2 + \alpha\|\nabla\phi_n\|^2 - \frac{\delta_0^2}{4c_0\eta}\|\nabla\theta_n\|^2 - c_0\eta\|\nabla\phi_n\|^2  \notag \\
& = \left(c_0(1-3\eta)+\alpha\right)\|\nabla\phi_n\|^2 - \frac{d_J^2}{2c_0\eta}\|\phi_n\|^2 - \frac{\delta_0^2}{4c_0\eta}\|\nabla\theta_n\|^2.  \notag
\end{align}
Together, these two yield
\begin{align}
\left(c_0(1-4\eta)+\alpha\right)\|\nabla\phi_n\|^2 - \frac{d_J^2}{2c_0\eta}\|\phi_n\|^2 - \frac{\delta_0^2}{4c_0\eta}\|\nabla\theta_n\|^2 \le \frac{1}{4c_0\eta}\|\nabla\mu_n\|^2,  \notag 
\end{align}
and with \eqref{wk-15}, \eqref{wk-17} and \eqref{wk-18} we deduce \eqref{wk-20.01}.

Now we seek a uniform bound for $\langle\mu_n\rangle$ in $L^2(0,T;H)$ so that we may bound $\mu_n$ uniformly in $L^2(0,T;V)$ (by virtue of \eqref{H1-norm}).
A simple estimate with \eqref{Fcons-3.1} shows,
\begin{align}
\langle\mu_n\rangle & = \langle a\phi_n \rangle - \langle P_n(J*\phi_n + F'(\phi_n)) \rangle + \alpha \langle \phi'_n \rangle - \delta \langle \theta_n \rangle  \notag \\ 
& = \frac{1}{|\Omega|}(a,\phi_n) - \frac{1}{|\Omega|}(P_n(J*\phi_n),1) + \frac{1}{|\Omega|}(P_n(F'(\phi_n)),1) + \frac{\alpha}{|\Omega|}(\phi'_n,1) - \frac{\delta}{|\Omega|}(\theta_n,1)  \notag \\
& \le \frac{1}{|\Omega|}\|J*1\|\|\phi_n\| + \frac{1}{|\Omega|^{1/2}}\|J*\phi_n\| + \frac{1}{|\Omega|}\|F'(\phi_n)\|_{L^1(\Omega)} + \frac{\alpha}{|\Omega|^{1/2}}\|\phi'_n\| + \frac{\delta_0}{|\Omega|^{1/2}}\|\theta_n\|  \notag \\
& \le \frac{2c_J}{|\Omega|^{1/2}}\|\phi_n\| + \frac{c_3}{|\Omega|}\|F(\phi_n)\|_{L^1(\Omega)} + c_4|\Omega| + \frac{\alpha}{|\Omega|^{1/2}}\|\phi'_n\| + \frac{\delta_0}{|\Omega|^{1/2}}\|\theta_n\|.  \label{mean-mu}
\end{align}
The desired bound now follows because of the uniform bounds in \eqref{wk-15}, and \eqref{wk-18}-\eqref{wk-20}.
Thus, we have shown
\begin{eqnarray}
\mu_n & ~\text{is uniformly bounded in}~ & L^2(0,T;V),  \label{wk-20.02} \\
F'(\phi_n) & ~\text{is uniformly bounded in}~ & L^\infty(0,T;L^1(\Omega)).  \label{wk-20.03}
\end{eqnarray}
Moreover, directly from \eqref{wk-20.02} and the discretized equation
\begin{align}
(\phi_n',\varphi) = -(\nabla\mu_n,\nabla\varphi),  \notag
\end{align}
we also have,
\begin{eqnarray}
\phi'_n & ~\text{is uniformly bounded in}~ & L^2(0,T;V').  \label{wk-20.04}
\end{eqnarray}

Next we obtain a bound for $\sqrt{\alpha}\phi_n$.
Indeed, we take $\varphi=\phi_n$ in \eqref{wk-3} to obtain,
\begin{align}
& \frac{1}{2}\frac{d}{dt} \left\{ \|\phi_n\|^2+\alpha\|\nabla\phi_n\|^2 \right\} + ((\nabla a)\phi_n+a\nabla\phi_n-\nabla J*\phi_n+F''(\phi_n)\nabla\phi_n,\nabla\phi_n)  \notag \\ 
& = \delta(\nabla\theta_n,\nabla\phi_n).  \label{wk-20.1}
\end{align}
Since $a(x)=\int_\Omega J(x-y)dy=J*1$, Young's inequality for convolutions shows 
\[
\|\nabla a\|_{L^\infty(\Omega)}=\|\nabla J*1\|_{L^\infty(\Omega)}\le \|\nabla J\|_{L^1(\Omega)} \|1\|_{L^\infty(\Omega)}.
\]
So we estimate (recall \eqref{not-1})
\begin{align}
((\nabla a)\phi_n & +a\nabla\phi_n-\nabla J*\phi_n+F''(\phi_n)\nabla\phi_n,\nabla\phi_n)  \notag \\ 
& = ((a+F''(\phi_n))\nabla\phi_n,\nabla\phi_n) + ((\nabla a)\phi_n-\nabla J*\phi_n,\nabla\phi_n)  \notag \\ 
& \ge c_0\|\nabla\phi_n\|^2 - \|\nabla a\|_{L^\infty(\Omega)}\|\nabla\phi_n\|\|\phi_n\| - \|\nabla J\|_{L^1(\Omega)}\|\nabla\phi_n\|\|\phi_n\|  \notag \\ 
& \ge c_0\|\nabla\phi_n\|^2 - 2d_J\|\nabla\phi_n\|\|\phi_n\|  \notag \\ 
& \ge \frac{c_0}{2}\|\nabla\phi_n\|^2 - \frac{2d_J^2}{c_0}\|\phi_n\|^2,  \label{wk-21}
\end{align}
and we use the basic estimate,
\begin{align}
\delta(\nabla\theta_n,\nabla\phi_n) \le \frac{\delta^2_0}{c_0}\|\nabla\theta_n\|^2 + \frac{c_0}{4}\|\nabla\phi_n\|^2.  \label{wk-22}
\end{align}
Together \eqref{wk-20.1}-\eqref{wk-22} produce,
\begin{align}
& \frac{d}{dt} \left\{ \|\phi_n\|^2+\alpha\|\nabla\phi_n\|^2 \right\} + \frac{c_0}{2}\|\nabla\phi_n\|^2 \le \frac{4d_J^2}{c_0}\|\phi_n\|^2 + \frac{2\delta_0^2}{c_0}\|\nabla\theta_n\|^2.  \label{wk-23}
\end{align}
Utilizing the bounds \eqref{wk-15} and \eqref{wk-18} following \eqref{wk-14}, and the definition of the $V$ norm \eqref{H1-norm}, we integrate \eqref{wk-23} with respect to $t$ over $(0,T)$ to find,
\begin{align}
& \|\phi_n(t)\|^2 + \alpha\|\phi_n(t)\|^2_{V} + \frac{c_0}{2} \int_0^t \|\nabla\phi_n(s)\|^2 ds  \notag \\ 
& \le \frac{4d_J^2}{c_0} \int_0^t \|\phi_n(s)\|^2 ds + \frac{2\delta_0^2}{c_0} \int_0^t \|\nabla\theta_n(s)\|^2 ds + \|\phi_{0n}\|^2 + \alpha\|\nabla\phi_{0n}\|^2 + \alpha|\langle \phi_{0n} \rangle|^2  \notag \\
& \le Q(\|\zeta_0\|_{\mathbb{H}^{\alpha,\ep}_m},T) + \|\nabla\phi_{0n}\|^2.  \label{wk-24}
\end{align}
This estimate implies 
\begin{eqnarray}
\sqrt{\alpha}\phi_n & ~\text{is uniformly bounded in}~ & L^\infty(0,T;V).  \label{wk-25}
\end{eqnarray}

We use the above results to bound $\theta'_n$.
Let us choose $\vartheta=\theta'_n$ in \eqref{wk-4}, which yields
\begin{align}
\frac{1}{2}\frac{d}{dt}\|\theta_n\|^2 + \|\theta'_n\|^2 & = -\delta(\phi'_n,\theta'_n)  \notag \\ 
& \le \frac{1}{2}\|\phi'_n\|^2 + \frac{\delta_0^2}{2}\|\theta'_n\|.  \notag
\end{align}
Integration over $(0,T)$ and the bounds \eqref{wk-16} and \eqref{wk-20} shows us that,
\begin{align*}
\|\theta_n(t)\|^2 + \int_0^t\|\theta'_n(s)\|^2 ds & \le \int_0^t \|\phi'_n(s)\|^2 ds + \delta_0^2 \int_0^t \|\theta'_n(s)\|^2 ds + \|\theta_{0n}\|^2  \notag \\ 
& \le \frac{1}{\nu_0}Q(\|\zeta_0\|_{\mathbb{H}^{\alpha,\ep}_m}),
\end{align*}
and hence,
\begin{eqnarray}
\theta'_n & ~\text{is uniformly bounded in}~ & L^2(0,T;H).  \label{wk-26}
\end{eqnarray}

Finally, we provide a bound for $\{\rho(\cdot,\phi_n)\}$.
Using (H4) again (see \eqref{Fcons-3.1}), we easily find, for any $p\in(1,2],$
\begin{align}
\|\rho(\cdot,\phi_n)\|_{L^p(\Omega)} & \le \|a\|_{L^\infty(\Omega)} \|\phi_n\| + \|F'(\phi_n)\|_{L^p(\Omega)}  \notag \\ 
& \le a^*\|\phi_n\| + \left( c_3\int_\Omega|F(\phi_n)|dx+c_4|\Omega| \right)^{1/p}.  \label{wk-26.1}
\end{align}
Employing \eqref{wk-15} and \eqref{wk-20}, it follows from \eqref{wk-26.1} that
\begin{eqnarray}
\rho(\cdot,\phi_n) & \in & L^\infty(0,T;L^p(\Omega)).  \label{wk-26.2}
\end{eqnarray}
This concludes Step 1.

{\em{Step 2. (Convergence of approximate solutions)}} In this step we pass to the limit to show that Problem P$_{\alpha,\ep}$ has a solution in the distributional sense, then we argue by density that this solution satisfies the identities for all appropriate test functions.
From the uniform bounds \eqref{wk-15}, \eqref{wk-16}, \eqref{wk-18}, \eqref{wk-19}, \eqref{wk-20.01}, \eqref{wk-20.02}, \eqref{wk-20.04}, \eqref{wk-25}, \eqref{wk-26}, and \eqref{wk-26.2}, by Alaoglu's
theorem (cf. e.g. \cite[Theorem 6.64]{Renardy&Rogers04}) there is a subsequence of $(\phi_n,\theta_n)^{tr}$ (generally not relabeled) and functions
\begin{eqnarray}
\phi & \in & L^\infty(0,T;H) \cap L^2(0,T;V),  \label{wk-27} \\ 
\sqrt{\alpha}\phi & \in & L^\infty(0,T;V),  \label{wk-27.1} \\
\theta & \in & L^\infty(0,T;H) \cap L^2(0,T;V),  \label{wk-28} \\ 
\mu & \in & L^2(0,T;V),  \label{wk-29} \\
\rho & \in & L^\infty(0,T;L^p(\Omega)),  \label{wk-29.1}
\end{eqnarray}
and
\begin{eqnarray}
\phi' & \in & L^2(0,T;V'),  \label{wk-30} \\ 
\sqrt{\alpha}\phi' & \in & L^2(0,T;H),  \label{wk-30.1} \\ 
\theta' & \in & L^2(0,T;H),  \label{wk-31} 
\end{eqnarray}
such that, as $n\rightarrow\infty$, 
\begin{eqnarray}
\phi_n \rightharpoonup \phi & \text{weakly-* in} & L^\infty(0,T;H),
\label{wk-32} \\
\phi_n \rightharpoonup \phi & \text{weakly in} & L^2(0,T;V),
\label{wk-33} \\
\sqrt{\alpha}\phi_n \rightharpoonup \sqrt{\alpha}\phi & \text{weakly-* in} & L^\infty(0,T;V),
\label{wk-33.5} \\
\theta_n \rightharpoonup \theta & \text{weakly-* in} & L^\infty(0,T;H),
\label{wk-34} \\
\theta_n \rightharpoonup \theta & \text{weakly in} & L^2(0,T;V),
\label{wk-35} \\
\mu_n \rightharpoonup \mu & \text{weakly in} & L^2(0,T;V), \label{wk-36} \\
\rho_n \rightharpoonup \rho & \text{weakly-* in} & L^\infty(0,T;L^p(\Omega)), \label{wk-36.1}  
\end{eqnarray}
and 
\begin{eqnarray}
\phi'_n \rightharpoonup \phi_t & \text{weakly in} & L^2(0,T;V'),  \label{wk-37} \\
\sqrt{\alpha}\phi'_n \rightharpoonup \sqrt{\alpha}\phi_t & \text{weakly in} & L^2(0,T;H),  \label{wk-37.5} \\
\theta'_n \rightharpoonup \theta_t & \text{weakly in} & L^2(0,T;H). \label{wk-38}  
\end{eqnarray}
Additionally, on account of the Aubin-Lions (compact) embedding (cf. e.g. \cite[Theorem 3.1.1]{Zheng04}),
\[
\{ \chi\in L^2(0,T;V), \ \chi_t\in L^2(0,T;V') \} \hookrightarrow L^2(0,T;H),
\]
we have
\begin{eqnarray}
\phi_n \rightarrow \phi & \text{strongly in} & L^2(0,T;H),
\label{wk-39} \\
\theta_n \rightarrow \theta & \text{strongly in} & L^2(0,T;H).
\label{wk-40}
\end{eqnarray}
An immediate consequence of \eqref{wk-39} is 
\begin{eqnarray}
J*\phi_n \rightarrow J*\phi & \text{strongly in} & L^2(0,T;V),
\label{wk-40.1}
\end{eqnarray}

We are now in position to pass to the limit in \eqref{wk-3}-\eqref{wk-8} to show that $\phi$, $\theta$, $\mu$, and $\rho$ satisfy \eqref{rel-2} and \eqref{ws-6}-\eqref{ws-9}.
To begin, using the pointwise convergence in \eqref{wk-39} and the (sequential) continuity assumption on $F$ in (H2), we immediately find
\begin{align}
\rho_n \rightarrow a\phi + F'(\phi) \quad \text{a.e. in $\Omega\times(0,T).$}  \label{wk-41}
\end{align}
Thanks to \eqref{wk-36.1}, we have \eqref{ws-6}; i.e.,
\[
\rho=a\phi+F'(\phi).
\]
Since 
\begin{equation*}
\mu_n = P_n(\rho_n - J*\phi_n + \alpha\phi'_n - \delta\theta_n)
\end{equation*}
then, for every $\varphi\in\Psi_j$, every $k\in\{1,\dots,j\}$ with $j\ge1$ fixed, and for every $\chi\in C_0^\infty((0,T)),$ there holds 
\begin{align}
\int_0^T (\mu_n(t),\varphi)\chi(t) dt = \int_0^T \left(\rho_n(t) - J*\phi_n(t) + \alpha\phi'_n(t) - \delta\theta_n(t),\varphi\right)\chi(t) dt.  \label{wk-42}
\end{align}
Letting $n\rightarrow+\infty$ in \eqref{wk-42}, whereby using \eqref{wk-36}, \eqref{wk-36.1}, \eqref{wk-37.5}, \eqref{wk-40} and \eqref{wk-40.1}, and by the density of $\Psi_\infty$ in $H$, we arrive at the equality \eqref{rel-2}; indeed,
\begin{align}
\mu & = \rho - J*\phi + \alpha\phi' + \delta\theta  \notag \\ 
& = a\phi + F'(\phi) - J*\phi + \alpha\phi_t + \delta\theta  \notag
\end{align}
holds in the distributional sense.
Moreover, we may update \eqref{wk-29.1} to include
\begin{eqnarray}
\rho & \in & L^2(0,T;H).  \label{wk-44}
\end{eqnarray}

We now show \eqref{ws-7} and \eqref{ws-8} hold.
To this end, multiply \eqref{wk-3} and \eqref{wk-4} by $\chi\in C_0^\infty((0,T))$ and $\omega\in C_0^\infty((0,T))$, respectively, and integrate with respect to $t$ over $(0,T)$.
This yields,
\begin{align}
\int_0^T (\phi_n'(t),\varphi) \chi(t) dt & + \int_0^T (\rho_n(t),-\Delta\varphi) \chi(t) dt - \int_0^T (\nabla J*\phi_n(t),\nabla\varphi) \chi(t) dt  \notag \\
& + \alpha\int_0^T (\phi'_n(t),-\Delta\varphi) \chi(t) dt = \delta\int_0^T (\nabla\theta_n(t),\nabla\varphi) \chi(t) dt,  \label{wk-45}
\end{align}
and
\begin{align}
\ep\int_0^T (\theta_n'(t),\vartheta) \omega(t) dt + \int_0^T (\nabla\theta_n(t),\nabla\vartheta) \omega(t) dt & = -\delta \int_0^T (\phi_n'(t),\vartheta) \omega(t) dt.  \label{wk-46}
\end{align}
On \eqref{wk-45} we pass to the limit $n\rightarrow+\infty$ using \eqref{wk-35}, \eqref{wk-36.1}, \eqref{wk-37} and \eqref{wk-40.1} to arrive at \eqref{ws-7} for every $\varphi\in V$ by virtue of a standard density argument.
Similarly, from \eqref{wk-46} for every $\vartheta\in V$, we gain \eqref{ws-8} using \eqref{wk-35}, \eqref{wk-37} and \eqref{wk-38}.

To show \eqref{ws-9} hold, we integrate \eqref{wk-3} and \eqref{wk-4} over $(0,t)$ and pass to the limit $n\rightarrow+\infty$.
This finishes Step 2.

{\em{Step 3. (Energy identity)}} To begin, let $\phi_0\in D(A_N)$, $\theta_0\in H$ and let $\zeta=(\phi,\theta)^{tr}$ be the corresponding weak solution.
Recall from \eqref{wk-39}, we have for almost any $t\in(0,T),$
\begin{align}
\phi_n(t)\rightarrow\phi(t) \quad \text{strongly in $H$ and a.e. in $\Omega.$}  \label{wk-50}
\end{align}
Since $F$ is measurable (see (H3)), Fatou's lemma implies
\begin{equation}  \label{wk-51}
\int_\Omega F(\phi(t)) dx \le \liminf_{n\rightarrow+\infty}\int_\Omega F(\phi_n(t)) dx.
\end{equation}
Additionally, thanks to \eqref{wk-40.1} and the fact that $P_n\in \mathcal{L}(V,V)$, then 
\begin{equation}  \label{wk-52}
P_n(J*\phi_n) \rightarrow J*\phi \quad \text{in $L^2(0,T;V)$}.
\end{equation}
Integrating \eqref{wk-12} on $(0,t)$, and passing to the limit while keeping in mind \eqref{wk-50}-\eqref{wk-52} and \eqref{wk-35}, \eqref{wk-36} and \eqref{wk-37.5}, as well as the weak lower-semicontinuity of the norm, we arrive at the differential inequality
\begin{align}
& \mathcal{E}_\ep(t) + \int_0^t \left( \|\nabla\mu(s)\|^2 + \alpha\|\phi_t(s)\|^2 + \|\nabla\theta(s)\|^2 \right) ds \le \mathcal{E}_\ep(0).  \label{wk-60}
\end{align}

We now show the energy equality \eqref{wk-2} holds.
The proof is based on the proof of \cite[Corollary 2]{CFG12}.
Here we require the regularity given in (H5).
Indeed, take $\varphi=\mu$ in \eqref{ws-7}.
Because of \eqref{ws-3}, we find the product $\langle \phi_t,\mu \rangle$ must contain the dual pairing $\langle \phi_t,F'(\phi) \rangle$.
It is here where we employ \eqref{convex} where $G$ is monotone increasing.
Now define the functional $\mathcal{G}:H\rightarrow\mathbb{R}$ by
\[
\mathcal{G}(\phi):=\left\{ \begin{array}{ll} \displaystyle\int_\Omega G(\phi)dx & \text{if}\ \mathcal{G}(\phi)\in L^1(\Omega), \\ +\infty & \text{otherwise}. \end{array} \right.
\]
Now by \cite[Proposition 2.8, Chapter II]{Barbu76}, it follows that $\mathcal{G}$ is convex, lower semicontinuous on $H$, and $\xi\in\partial\mathcal{G}(\phi)$ if and only if $\xi=\mathcal{G}'(\phi) = G(\phi)$ almost everywhere in $\Omega$.
Applying \cite[Proposition 4.2]{CKRS07}, also, for almost all $t\in(0,T),$
\begin{align}
\langle \phi_t,F'(\phi) \rangle & = \langle \phi_t,G(\phi)\rangle - a^*\langle \phi_t,\phi\rangle  \notag \\ 
& = \frac{d}{dt} \left\{ \mathcal{G}(\phi)-\frac{a^*}{2}\|\phi\|^2 \right\}  \notag \\ 
& = \frac{d}{dt} \int_\Omega F(\phi)dx.
\end{align}
Hence, 
\begin{align}
\frac{1}{2}\frac{d}{dt} & \left\{ \|\sqrt{a}\phi\|^2-(\phi,J*\phi)+\int_\Omega F(\phi)dx \right\}+\alpha\|\phi_t\|^2-\delta\langle\phi_t,\theta\rangle+\|\nabla\mu\|^2=0.  \notag
\end{align}
Next we add in the identity obtained after taking $\vartheta=\theta$ in \eqref{ws-8} and apply \eqref{wk-1} to find 
\begin{align}
& \frac{d}{dt} \left\{ \frac{1}{4}\int_\Omega\int_\Omega J(x-y)(\phi(x)-\phi(y))^2dxdy + \int_\Omega F(\phi)dx+\frac{\ep}{2}\|\theta\|^2\right\}  \notag \\ 
&+\alpha\|\phi_t\|^2+\|\nabla\theta\|^2+\|\nabla\mu\|^2=0.  \notag 
\end{align}
Integrating this differential identity on $(0,t)$ produces \eqref{wk-2} as claimed.
This concludes Step 3.

{\em{Step 4. (Assuming $\phi_0\in H$ is such that $F(\phi_0)\in L^1(\Omega)$, and the continuity conditions)}} 
Take $\zeta_0=(\phi_0,\theta_0)^{tr}\in \mathbb{H}^{\alpha,\ep}_m$ where $F(\phi_0)\in L^1(\Omega).$
Proceeding exactly as in \cite[page 440]{CFG12} the bounds \eqref{wk-15}-\eqref{wk-20.01}, \eqref{wk-20.02}-\eqref{wk-20.04} and \eqref{wk-25}, \eqref{wk-26} and \eqref{wk-26.2} hold.
Then applying the Aubin-Lions compactness embedding and find $\phi$, $\theta,$ $\mu$, and $\rho$ that satisfy \eqref{wk-27}-\eqref{wk-29.1}.
Passing to the limit in the variational formulation for $\zeta_k=(\phi_k,\theta)^{tr}$, we find the pair $\zeta=(\phi,\theta)^{tr}$ is a solution corresponding to the initial data $\zeta_0=(\phi_0,\theta_0)^{tr}\in \mathbb{H}^{\alpha,\ep}_m$ for which $F(\phi_0)\in L^1(\Omega).$

Finally, the continuity properties
\begin{equation}  \label{continuity-1}
\phi\in C([0,T];H) \quad\text{and}\quad \theta\in C([0,T];H)
\end{equation}
follow from the classical embedding (cf. \cite[Lemma 5.51]{Tanabe79}),
\[
\{ \chi\in L^2(0,T;V), \ \chi_t\in L^2(0,T;V') \} \hookrightarrow C([0,T];H),
\] 
and the conditions \eqref{wk-27}$_2$ and \eqref{wk-28}$_2$ with \eqref{wk-30} and \eqref{wk-31} established above.
This finishes the proof of the theorem.
\end{proof}

\begin{remark}  \label{r:gradient}
From \eqref{wk-2} we see that if there is a $t_0>0$ in which 
\[
\mathcal{E}_\ep(t_0)=\mathcal{E}_\ep(0),
\]
then, for all $t\in(0,t_0),$ 
\begin{equation}  \label{no-grad}
\int_0^{t} \left( \|\nabla\mu(s)\|^2 + \alpha\|\phi_t(s)\|^2 + \|\nabla\theta(s)\|^2 \right)ds = 0.
\end{equation}
Hence, we deduce $\phi_t(t)=0$ and $\theta_t(t)=0$ for all $t\in(0,t_0)$.
Therefore, $\zeta=(\phi,\theta)^{tr}$ is a fixed point of the trajectory $\zeta(t)=S_{\alpha,\ep}(t)\zeta_0$.
Since the semigroup $S_{\alpha,\ep}(t)$ is precompact (per Lemma \ref{t:thermal-reg} and Lemma \ref{t:pseudometric}), the system $(\mathcal{X}^{\alpha,\ep}_m,S_{\alpha,\ep},\mathcal{E}_\ep)$ is gradient/conservative for each $\alpha\in(0,1]$ and $\ep\in(0,1]$.
\end{remark}

The following proposition establishes the uniqueness of weak solutions to Problem P$_{\alpha,\ep}$. 
Furthermore, it shows that the semigroup $S_{\alpha,\ep}$ (defined below) is strongly continuous with respect to the metric $\mathcal{X}^{\alpha,\ep}_m$.

\begin{proposition}  \label{t:cont-dep}
Assume (H1)-(H4) hold. 
Let $T>0$, $m\ge0$, $\delta_0>0$, $\delta\in(0,\delta_0]$, $(\alpha,\ep)\in(0,1]\times(0,1]$, and $\zeta_{01}=(\phi_{01},\theta_{01})^{tr}$, $\zeta_{02}=(\phi_{02},\theta_{02})^{tr}\in\mathbb{H}^{\alpha,\ep}_m$ be such that $F(\phi_{01}),F(\phi_{02})\in L^1(\Omega)$.
Let $\zeta_1(t)=(\phi_1(t),\theta_1(t))$ and $\zeta_2(t)=(\phi_2(t),\theta_2(t))$ denote the weak solution to Problem P$_{\alpha,\ep}$ corresponding to the data $\zeta_{01}$ and $\zeta_{02}$, respectively.
Then there are positive constants $\bar\nu_1=\bar\nu_1(c_0,J,\alpha,\ep,\delta_0)\sim\{\alpha^{-2},\ep^{-1}\}$ and $\bar\nu_2=\bar\nu_2(F,J,\Omega,\delta_0)$, independent of $T$, $\zeta_{01}$, and $\zeta_{02}$, such that, for all $t\in[0,T],$ 
\begin{align}
& \|\zeta_{1}(t)-\zeta_{2}(t)\|^2_{\mathbb{H}^{\alpha,\ep}_m} + \int_0^t \left( 2\|\partial_t\phi_1(s)-\partial_t\phi_2(s)\|^2_{V'} + \alpha\|\partial_t\phi_1(s)-\partial_t\phi_2(s)\|^2 + 2\|\theta_1(s)-\theta_2(s)\|^2_{V} \right) ds  \notag \\ 
& \le e^{\bar\nu_1 t} \left( \|\zeta_1(0)-\zeta_2(0)\|^2_{\mathbb{H}^{\alpha,\ep}_m} + \frac{2\bar\nu_2}{\bar\nu_1} \left( |M_1-M_2|+|N_1-N_2| \right)^2  \right)  \label{diff-0}
\end{align}
where $M_i:=\langle\phi_i(0)\rangle$, $N_i:=\langle\theta_i(0)\rangle$, $i=1,2$.
\end{proposition}

\begin{proof}
We see that upon setting, for all $t\in[0,T],$
\begin{align}
\bar\zeta(t) = (\bar\phi(t),\bar\theta(t)) & := (\phi_1(t),\theta_1(t))-(\phi_2(t),\theta_2(t))  \notag \\ 
& = \zeta_1(t)-\zeta_2(t),  \notag
\end{align}
the difference $\bar\zeta=(\bar\phi,\bar\theta)$ formally satisfies the equations 
\begin{eqnarray}
\bar\phi_t = \Delta\bar\mu &\text{in}& \Omega\times(0,T)  \label{diff-1} \\ 
\bar\mu = a\bar\phi - J*\bar\phi + F'(\phi_1) - F'(\phi_2) + \alpha \bar\phi_t - \delta\bar\theta &\text{in}& \Omega\times(0,T)  \label{diff-2} \\ 
\ep\bar\theta_t - \Delta\bar\theta = -\delta \bar\phi_t &\text{in}& \Omega\times(0,T)  \label{diff-3} \\ 
\partial_n\bar\mu = 0 &\text{on}& \Gamma\times(0,T)  \label{diff-4} \\ 
\partial_n\bar\theta = 0 &\text{on}& \Gamma\times(0,T)  \label{diff-5} \\ 
\bar\phi(x,0) = \phi_{01}(x)-\phi_{02}(x) &\text{at}& \Omega\times\{0\}  \label{diff-6} \\ 
\bar\theta(x,0) = \theta_{01}(x)-\theta_{02}(x) &\text{at}& \Omega\times\{0\}. \label{diff-7}  
\end{eqnarray}
Multiply \eqref{diff-1}-\eqref{diff-3} by, $A_N^{-1}(\bar\phi_t+\bar\phi-\langle\bar\phi\rangle)$, $\bar\phi_t+\bar\phi$, and $\bar\theta$, respectively (notice that, by \eqref{con-mass} $\bar\phi_t+\bar\phi-\langle\bar\phi\rangle\in V'_0$), and sum the resulting identities to yield, for almost all $t\in[0,T],$
\begin{align}
& \frac{d}{dt} \left\{ \|\bar\phi\|^2_{V'} + \alpha\|\bar\phi\|^2 + \ep\|\bar\theta\|^2 \right\} + 2\|\bar\phi_t\|^2_{V'} + 2\alpha\|\bar\phi_t\|^2 + 2\|\bar\theta\|^2_{V}  \notag \\ 
& + 2(a\bar\phi+F'(\phi_1)-F'(\phi_2),\bar\phi_t+\bar\phi) - 2(J*\bar\phi,\bar\phi_t+\bar\phi) - 2\delta(\bar\theta,\bar\phi)  \notag \\ 
& = 2|\Omega|\langle \bar\phi \rangle \langle \bar\mu \rangle + 2\langle\bar\theta\rangle^2.  \label{diff-9}
\end{align}
Estimating the resulting products using assumption (H2) yields, 
\begin{align}
2(a\bar\phi + F'(\phi_1) - F'(\phi_2),\bar\phi_t + \bar\phi) & \ge 2c_0(\bar\phi,\bar\phi_t) + 2c_0\|\bar\phi\|^2  \notag \\ 
& \ge 2c_0\left( 1-\frac{c_0}{\alpha} \right)\|\bar\phi\|^2 - \frac{\alpha}{2}\|\bar\phi_t\|^2,  \label{diff-11}
\end{align}
recalling $J\in W^{1,1}(\mathbb{R}^3)$ and following \cite[Proposition 5, (4.2) and (4.3)]{Frigeri&Grasselli12} we also write,
\begin{align}
2|(J*\bar\phi,\bar\phi)| & \le 2\|A_N^{1/2}(J*\bar\phi)\|\|A_N^{-1/2}\bar\phi\|  \notag \\ 
& \le d_J^2\|\bar\phi\|^2 + \|\bar\phi\|^2_{V'} + C|\Omega||\langle\bar\phi\rangle|,  \label{diff-12}
\end{align}
\begin{align}
-2(J*\bar\phi,\bar\phi_t) & \ge -2\|J*\bar\phi\|\|\bar\phi_t\|  \notag \\ 
& = -\frac{2c_J^2}{\alpha}\|\bar\phi\|^2 - \frac{\alpha}{2}\|\bar\phi_t\|^2,  \label{diff-13}
\end{align}
and
\begin{align}
-2\delta(\bar\theta,\bar\phi) & \ge -\delta_0^2\|\bar\theta\|^2 - \|\bar\phi\|^2.  \label{diff-14}
\end{align}
Combining \eqref{diff-9}-\eqref{diff-14}, we have, for almost all $t\in[0,T],$
\begin{align}
& \frac{d}{dt} \left\{ \|\bar\phi\|^2_{V'} + \alpha\|\bar\phi\|^2 + \ep\|\bar\theta\|^2 \right\} + 2\|\bar\phi_t\|^2_{V'} + \alpha\|\bar\phi_t\|^2 + 2\|\bar\theta\|^2_{V}  \notag \\ 
& \le \|\bar\phi\|^2_{V'} + \frac{1}{\alpha}\left( 2c_0\left( \frac{c_0}{\alpha}-1 \right) + d_J^2 + \frac{2c_J^2}{\alpha} + 1 \right)\cdot\alpha\|\bar\phi\|^2 + \frac{\delta^2_0}{\ep}\cdot\ep\|\bar\theta\|^2  \notag \\ 
& + C|\Omega|\langle \bar\phi \rangle \langle \bar\mu \rangle + 2\langle\bar\theta\rangle^2.  \label{diff-14.9}
\end{align}
We readily find that there is a constant,
\begin{equation}  \label{diff-15}
\bar\nu_1=\bar\nu_1(c_0,c_J,d_J,\alpha,\ep,\delta_0):=\max\left\{ 1,\frac{1}{\alpha}\left( 2c_0\left( \frac{c_0}{\alpha}-1 \right) + d_J^2 + \frac{2c_J^2}{\alpha} + 1 \right),\frac{\delta_0^2}{\ep} \right\}\ge1,
\end{equation}
and, using the local Lipschitz assumption (H2), it is easy to show that, 
\begin{align}
|\langle \bar\mu \rangle| & \le C_F|\langle \bar\phi \rangle| + \delta_0|\langle \bar\theta \rangle|  \notag \\ 
& =:\bar\mu_*,  \label{diff-19.1}
\end{align}
for some positive constant $C_F$ depending on $c_J$ and the Lipschitz bound on $F'$.
Thus, \eqref{diff-14.9} becomes, for almost all $t\in[0,T],$ 
\begin{align}
& \frac{d}{dt} \left\{ \|\bar\phi\|^2_{V'} + \alpha\|\bar\phi\|^2 + \ep\|\bar\theta\|^2 \right\} + 2\|\bar\phi_t\|^2_{V'} + \alpha\|\bar\phi_t\|^2 + 2\|\bar\theta\|^2_{V}  \notag \\ 
& \le \bar\nu_1\left( \|\bar\phi\|^2_{V'} + \alpha\|\bar\phi\|^2 + \ep\|\bar\theta\|^2 \right) + C|\Omega||\langle \bar\phi \rangle| \bar\mu_* + 2\langle \bar\theta \rangle^2.  \label{diff-18}
\end{align}
Integrating \eqref{diff-18} over $(0,t)$, we obtain, for all $t\in[0,T],$
\begin{align}
\|\bar\phi(t)\|^2_{V'} + \alpha\|\bar\phi(t)\|^2 + \ep\|\bar\theta(t)\|^2 & + \int_0^t \left( 2\|\bar\phi_t(s)\|^2_{V'} + \alpha\|\bar\phi_t(s)\|^2 + 2\|\bar\theta(s)\|^2_{V} \right) ds  \notag \\ 
& \le e^{\bar\nu_1t}\|\bar\zeta(0)\|^2_{\mathbb{H}^{\alpha,\ep}_m} + \frac{2}{\bar\nu_1}| \left( |\Omega||\langle \bar\phi \rangle| \bar\mu_* + |\langle \bar\theta \rangle|^2 \right) \left( e^{\bar\nu_1t}-1 \right)  \notag \\ 
& \le e^{\bar\nu_1 t} \left( \|\bar\zeta(0)\|^2_{\mathbb{H}^{\alpha,\ep}_m} + \frac{2\bar\nu_2}{\bar\nu_1} \left( |\langle \bar\phi \rangle|+|\langle \bar\theta \rangle| \right)^2  \right),  \label{diff-19}
\end{align}
where $\bar\nu_2=\bar\nu_2(F,J,\Omega,\delta_0):=\max\{ C_F,|\Omega|,\frac{\delta_0}{2}|\Omega|, 1 \}\ge 1$.
From \eqref{diff-19} we find the estimate \eqref{diff-0} holds.
This finishes the proof.
\end{proof}

As before, we can now formalize the semi-dynamical system generated by Problem P$_{\alpha,\ep}$.

\begin{corollary}  \label{t:Lip-semif}
Let the assumptions of Theorem \ref{t:existence} be satisfied. 
We can define a strongly continuous semigroup (of solution operators) $S_{\alpha,\ep}=(S_{\alpha,\ep}(t))_{t\ge0}$, for each $\alpha>0$ and $\ep>0$,
\begin{equation*}
S_{\alpha,\ep}(t):\mathcal{X}^{\alpha,\ep}_m\rightarrow \mathcal{X}^{\alpha,\ep}_m
\end{equation*}
by setting, for all $t\geq 0,$ 
\begin{equation*}
S_{\alpha,\ep}(t)\zeta_0:=\zeta(t)
\end{equation*}
where $\zeta(t)=(\phi(t),\theta(t))$ is the unique global weak solution to Problem P$_{\alpha,\ep}$.
Furthermore, as a consequence of \eqref{diff-0}, if we assume 
\[
M_1=M_2 \quad\text{and}\quad N_1=N_2,
\]
the semigroup $S_{\alpha,\ep}(t):\mathcal{X}^{\alpha,\ep}_m\rightarrow \mathcal{X}^{\alpha,\ep}_m$ is Lipschitz continuous on $\mathcal{X}^{\alpha,\ep}_m$, uniformly in $t$ on compact intervals. 
\end{corollary}

\subsection{Bounded absorbing sets for Problem P$_{\alpha,\ep}$}

We now give a dissipation estimate for Problem P$_{\alpha,\ep}$ from which we deduce the existence of an absorbing set.
The idea of the estimate follows \cite[Proposition 2]{Gal&Miranville09}.
It is here where we require the slight modification of hypothesis (H1).

\begin{lemma}  \label{t:diss-1}
Assume (H1)-(H4) hold.
Let $m\ge0$, $\delta_0>0$, $\delta\in(0,\delta_0]$, $(\alpha,\ep)\in(0,1]\times(0,1]$, $\zeta_0=(\phi_0,\theta_0)^{tr}\in \mathbb{H}^{\alpha,\ep}_m$ with $F(\phi_0)\in L^1(\Omega).$
Assume $\zeta=(\phi,\theta)^{tr}$ is a weak solution to Problem P$_{\alpha,\ep}$.
There is a positive constant $\nu_3=\nu_3(\delta_0,J,\Omega)$, but independent of $\alpha$, $\ep$, and $\zeta_0$, such that, for all $t\ge0$, the following holds, 
\begin{align}
& \|\hat\phi(t)\|^2_{V'} + \alpha\|\hat\phi(t)\|^2 + \|\sqrt{a}\phi(t)\|^2 + \|\hat\theta(t)\|^2 + (F(\phi(t)),1) - (J*\phi(t),\hat\phi(t))  \notag \\ 
& + \int_t^{t+1} \left( \|\phi_t(s)\|^2_{V'} + \alpha\|\phi_t(s)\|^2 + \|\theta(s)\|^2_{V} \right) ds  \notag \\
& \le Q(\|\zeta_0\|_{\mathbb{H}^{\alpha,\ep}_m})e^{-\nu_3t} + \frac{1}{\nu_3}Q(m),  \label{ap-1}
\end{align}
for some monotonically increasing functions $Q$.

Consequently, the set given by
\begin{equation}  \label{abs-set}
\mathcal{B}^{\alpha,\ep}_0 := \left\{ \zeta\in\mathbb{H}^{\alpha,\ep}_m : \|\zeta\|^2_{\mathbb{H}^{\alpha,\ep}_m} \le \frac{1}{\nu_3}Q(m)+1 \right\},
\end{equation}
where $Q(\cdot,\cdot)$ is the function from \eqref{ap-1}, is a closed, bounded absorbing set in $\mathbb{H}^{\alpha,\ep}_m$, positively invariant under the semigroup $S_{\alpha,\ep}$.
\end{lemma}

\begin{proof}
We give a formal calculation that can be justified by a suitable Faedo-Galerkin approximation based on the proof of Theorem \ref{t:existence} above.
Let $M_0:=\langle\phi_0\rangle$ and $N_0:=\langle\theta_0\rangle$.
Multiply \eqref{rel-1}-\eqref{rel-3} by, $A_N^{-1}\phi_t$, $\phi_t$, and $\hat\theta:=\theta-N_0$, respectively, then integrate over $\Omega$, applying \eqref{NLr-1} (since $\phi_t=\phi_t-\langle\phi_t\rangle$ belongs to $V_0'$; recall \eqref{con-mass}), and sum the resulting identities to arrive at the differential identity, which holds for almost all $t\ge0$,
\begin{align}
& \frac{d}{dt} \left\{ \|\sqrt{a}\phi\|^2 + \ep\|\hat\theta\|^2 + 2(F(\phi),1) - (J*\phi,\phi) \right\} + 2\|\phi_t\|^2_{V'} + 2\alpha\|\phi_t\|^2 + 2\|\nabla\theta\|^2 = 0.  \label{ap-2}
\end{align}
Let $\hat\phi:=\phi-M_0$.
We further multiply \eqref{rel-1}-\eqref{rel-2} by, $2\xi A_N^{-1}\hat\phi$ and $2\xi\hat\phi$, respectively, in $H$, where $\xi>0$ is to be determined below.
Observe $\langle\hat\phi\rangle=0$ and $\|\hat\phi\|^2=\|\phi\|^2-M_0^2|\Omega|.$
This yields, for almost all $t\ge0,$
\begin{align}
& \frac{d}{dt} \left\{ \xi\|\hat\phi\|^2_{V'} + \xi\alpha\|\hat\phi\|^2 \right\} + 2\xi\|\sqrt{a}\hat\phi\|^2 + 2\xi(F'(\phi),\hat\phi)  \notag \\ 
& = 2\xi(J*\phi,\hat\phi) + 2\xi\delta(\theta,\hat\phi) - 2\xi M_0(a,\hat\phi).  \label{ap-3}
\end{align}
Together, \eqref{ap-2} and \eqref{ap-3} make the differential identity,
\begin{align}
& \frac{d}{dt} \left\{ \xi\|\hat\phi\|^2_{V'} + \xi\alpha\|\hat\phi\|^2 + \|\sqrt{a}\phi\|^2 + \ep\|\hat\theta\|^2 + 2(F(\phi),1) - (J*\phi,\hat\phi) \right\}  \notag \\ 
& + 2\|\phi_t\|^2_{V'} + 2\alpha\|\phi_t\|^2 + 2\xi\|\sqrt{a}\hat\phi\|^2 + 2\|\nabla\theta\|^2 + 2\xi(F'(\phi),\hat\phi)  \notag \\
& = 2\xi(J*\phi,\hat\phi) + 2\xi\delta(\theta,\hat\phi) - 2\xi M_0(a,\hat\phi).  \label{ap-4}
\end{align}
Introduce the functional defined by, for all $t\ge0$ and $\xi>0$,
\begin{equation}  \label{ap-5}
E(t):=\xi\|\hat\phi(t)\|^2_{V'} + \xi\alpha\|\hat\phi(t)\|^2 + \|\sqrt{a}\phi(t)\|^2 + \ep\|\hat\theta(t)\|^2 + 2(F(\phi(t)),1) - (J*\phi,\hat\phi) + C_F.
\end{equation}
(Observe, $E(t) = 2\mathcal{E}_\ep(t) + \xi\|\hat\phi(t)\|^2_{V'} + \xi\alpha\|\hat\phi(t)\|^2 + C_F$.)
Because of assumption (H3) and the assumption that $F(\phi_0)\in L^1(\Omega)$, we know 
\begin{equation}  \label{Fcons-5}
2(F(\phi),1) - (J*\phi,\hat\phi) \ge (2c_1-2c_1)\|\hat\phi\|^2+2c_1 M_0^2|\Omega|-2c_2|\Omega|,
\end{equation}
thus the constant $C_F$ may be chosen sufficiently large to insure $E(t)$ is non-negative for all $t\ge0$, $\alpha\in(0,1]$, $\ep\in(0,1]$, and $\xi>0.$
Then we rewrite \eqref{ap-4} as, 
\begin{align}
& \frac{d}{dt} E + \tau E = H,  \label{ap-6}
\end{align}
for some $0<\tau<\xi,$ and where 
\begin{align}
H(t) & := \tau\xi\|\hat\phi(t)\|^2_{V'} + \tau\xi\alpha\|\hat\phi(t)\|^2 + \tau\|\sqrt{a}\phi(t)\|^2 + \tau\ep\|\hat\theta(t)\|^2 + 2\tau(F(\phi(t)),1) - \tau(J*\phi,\hat\phi) + \tau C_F  \notag \\ 
& - 2\|\phi_t(t)\|^2_{V'} - 2\alpha\|\phi_t(t)\|^2 - 2\xi\|\sqrt{a}\hat\phi(t)\|^2 - 2\|\nabla\theta(t)\|^2 - 2\xi(F'(\phi(t)),\hat\phi(t))  \notag \\
& + 2\xi(J*\phi(t),\hat\phi(t)) + 2\xi\delta(\theta(t),\hat\phi(t)) - 2\xi M_0(a,\hat\phi(t)).  \label{ap-7}
\end{align}
Estimating the products on the right-hand side using the assumptions (H1)-(H3) as well as Young's inequality for convolutions (cf. e.g. \cite[Corollary 2.25]{Adams&Fournier03}), and the Poincar\'{e}-type inequality \eqref{Poincare2} yields (and recall $\delta\in(0,\delta_0]$), 
\begin{align}
2\xi(J*\phi,\hat\phi) & \le 2\xi \|J*\phi\| \|\hat\phi \|  \notag \\ 
& \le 2\xi c_J \|\hat\phi\|^2 + M_0^2\|a\|_\infty^2 + \xi^2\|\hat\phi\|^2,  \label{ap-8}
\end{align}
\begin{align}
2\xi\delta(\theta,\hat\phi) & \le 2\xi\delta_0 \|\theta\|\|\hat\phi\|  \notag \\ 
& \le \xi\delta_0^2\|\theta\|^2 + \xi\|\hat\phi\|^2  \notag \\
& \le 2\xi\delta_0^2\lambda_\Omega\|\nabla\theta\|^2 + 2\xi\delta_0^2|\Omega|N_0^2 + \xi\|\hat\phi\|^2,  \label{ap-10}
\end{align}
and
\begin{align}
-2\xi M_0(a,\hat\phi) & \le 2\xi M_0 \|a\| \|\hat\phi\|  \notag \\ 
& = 2\xi M_0\|J*1\|\|\hat\phi\|  \notag \\
& \le 2\xi M_0 c_J |\Omega|^{1/2}\|\hat\phi\|  \notag \\
& \le M_0^2 c_J^2 |\Omega| + \xi^2 \|\hat\phi\|^2.  \label{ap-11}
\end{align}
With assumption (H3) we now consider, with the aid of \eqref{Fcons-1}-\eqref{Fcons-3} (setting $m=M_0$),
\begin{align}
2\tau(F(\phi),1) - 2\xi(F'(\phi),\hat\phi) & = -2\tau\left( (F'(\phi),\hat\phi)-(F(\phi),1) \right) - 2(\xi-\tau)(F'(\phi),\hat\phi)  \notag \\
& = -2\tau(F'(\phi)\hat\phi-F(\phi),1)-2(\xi-\tau)(F'(\phi),\hat\phi)  \notag \\ 
& \le 2\tau c_9|\Omega| + 2\tau c_{10}\|\hat\phi\|^2 - (\xi-\tau)(|F(\phi)|,1) + 2(\xi-\tau)c_{11} + (\xi-\tau)c_{12}.  \label{ap-12} 
\end{align}
By (H1) again, we find that for a fixed $0<a_0<\essinf_\Omega a(x)$ (this is where we need the slightly stricter version of (H1)), there holds 
\[
a_0\|\hat\phi\|^2 \le \|\sqrt{a}\hat\phi\|^2.
\]
Moreover, due to the continuous embedding $H\hookrightarrow V',$ there is a constant, which we denote $C_\Omega>0$, so that $C_\Omega^{-2}\|\hat\phi\|^2_{V'}\le\|\hat\phi\|^2$ (cf. e.g. \cite[p. 243, Equation (6.7)]{Milani&Koksch05}), and, now with $0<\xi<1,$ 
\begin{align}
-2\xi\|\sqrt{a}\hat\phi\|^2 & \le -\frac{a_0}{2} C^{-2}_\Omega\|\hat\phi\|^2_{V'} - \frac{a_0}{2}\|\hat\phi\|^2 - \xi\|\sqrt{a}\hat\phi\|^2.  \label{ap-13}
\end{align}
Also observe that, using the Poincar\'{e}-type inequality \eqref{Poincare} again, we have
\begin{align}
-\left( 2 - 2\xi\delta_0^2\lambda_\Omega \right)\|\nabla\theta\|^2 \le -\left( 1 - 2\xi\delta_0^2\lambda_\Omega \right)\|\nabla\theta\|^2 - \frac{1}{\lambda_\Omega}\|\hat\theta\|^2.  \label{ap-13.5}
\end{align} 
Combining \eqref{ap-7}-\eqref{ap-13.5} yields, 
\begin{align}
H & \le \left( \tau\xi - \frac{a_0}{2} C^{-2}_\Omega \right)\|\hat\phi\|^2_{V'} + \left( \tau\xi\alpha + 2\xi c_J + \xi + 2\xi^2 + 2\tau c_{10} - \frac{a_0}{2} \right)\|\hat\phi\|^2  \notag \\ 
& + \left( \tau - \frac{\xi}{2} \right) \|\sqrt{a}\phi\|^2 + \left( \tau\ep - \frac{1}{\lambda_\Omega} \right)\|\hat\theta\|^2 - (\xi-\tau)(|F(\phi)|,1)  \notag \\ 
& - 2\|\phi_t\|^2_{V'} - 2\alpha\|\phi_t\|^2 - \left(1-2\xi\delta_0^2\lambda_\Omega \right) \|\nabla\theta\|^2  \notag \\
& + \tau C_F + M_0^2 c_J^2|\Omega| + 2\xi\delta_0^2|\Omega|N_0^2 + 2\tau c_9|\Omega|  \notag \\
& + 2(\xi-\tau)c_{11} + (\xi-\tau)c_{12} + \xi M_0^2|\Omega|(\langle a \rangle - a_0) + M_0^2\|a\|_\infty^2.  \label{ap-14}
\end{align}
We should note that the additional constants in $a$ on the right-hand side of \eqref{ap-14} is due to the fact that
\[
-\xi\|\sqrt{a}\hat\phi\|^2 \ge -\xi\|\sqrt{a}\phi\|^2 - \xi M_0^2|\Omega|(\langle a \rangle - a_0).
\]
Inserting \eqref{ap-14} into \eqref{ap-6} produces the differential inequality (this is where we use the condition that $0<\alpha\le1$ and $0<\ep\le1$), 
\begin{align}
& \frac{d}{dt} E + 2\|\phi_t\|^2_{V'} + 2\alpha\|\phi_t\|^2 + \left( 1-2\xi\delta_0^2\lambda_\Omega \right)\|\theta\|^2_{V}  \notag \\ 
& + \frac{a_0}{4}C^{-2}_\Omega\|\hat\phi\|^2_{V'} + \left( \frac{a_0}{2} - 2\xi c_J - \xi - 2\xi^2 - 2\tau c_{10} \right)\alpha\|\hat\phi\|^2  \notag \\ 
& + \xi\|\sqrt{a}\phi\|^2 + \frac{1}{\lambda_\Omega}\ep\|\hat\theta\|^2 + (\xi - \tau)(F(\phi),1) + \tau C_F  \notag \\ 
& \le \tau C_F + M_0^2 c_J^2|\Omega| + 2\xi\delta_0^2|\Omega|N_0^2 + \left( 1-2\xi\delta_0^2\lambda_\Omega \right)N^2_0  \notag \\ 
& + 2\tau c_9|\Omega| + 2(\xi-\tau)c_{11} + (\xi-\tau)c_{12} + \xi M_0^2|\Omega|(\langle a \rangle - a_0) + M_0^2\|a\|_\infty^2.  \notag 
\end{align}
The extra term with $N_0$ now appearing on the right-hand side is used to make the $V$ norm in $\theta$.
Now we easily see that there are $0<\tau<\xi<1$ so that 
\[
\nu_3=\nu_3(\delta_0,J,\Omega):=\min\left\{ 1-2\xi\delta_0^2\lambda_\Omega, \frac{a_0}{2} - 2\xi c_J - \xi - 2\xi^2 - 2\tau c_{10} \right\}>0.
\]
Now there holds, for almost all $t\ge0,$
\begin{align}
& \frac{d}{dt} E + \nu_3 E + \|\phi_t\|^2_{V'} + 2\alpha\|\phi_t\|^2 + \nu_3\|\theta\|^2_{V} \le Q(m).  \label{ap-15}
\end{align}
Neglecting the normed terms $\|\phi_t\|^2_{V'} + 2\alpha\|\phi_t\|^2 + \nu_3\|\theta\|^2_{V}$, then employing a Gr\"{o}nwall inequality yields, for all $t\ge0,$
\begin{equation}
E(t) \le e^{-\nu_3 t}E(0) + \frac{1}{\nu_3}Q(m).  \label{ap-16}
\end{equation}
Recall that $F(\phi_0)\in L^1(\Omega)$ by assumption, so now we easily arrive at 
\begin{align}
& \|\hat\phi(t)\|^2_{V'} + \alpha\|\hat\phi(t)\|^2 + \|\sqrt{a}\phi(t)\|^2 + \|\hat\theta(t)\|^2 + (F(\phi(t)),1) - (J*\phi(t),\hat\phi(t))  \notag \\ 
& \le E(0)e^{-\nu_3t} + \frac{1}{\nu_3}Q(m).  \label{ap-17}
\end{align}
Also, by neglecting the positive term $\nu_3E$ in \eqref{ap-15} and integrating this time over $(t,t+1)$, we find, with \eqref{ap-16}, for all $t\ge0$,
\begin{equation}
\int_t^{t+1}\left( \|\phi_t(s)\|^2_{V'} + \alpha\|\phi_t(s)\|^2 + \|\theta(s)\|^2_{V} \right)ds \le E(0)e^{-\nu_3t} + \left( \frac{1}{\nu_3} + 1 \right)Q(m).  \label{ap-18}
\end{equation}
Together, \eqref{ap-17} and \eqref{ap-18} establish \eqref{ap-1}.

The existence of the set $\mathcal{B}^{\alpha,\ep}_0$ described in \eqref{abs-set} follows directly from the dissipation estimate \eqref{ap-1}; indeed, (cf. e.g. \cite{Babin&Vishik92}).
To see why $\mathcal{B}^{\alpha,\ep}_0$ is absorbing, consider any nonempty bounded subset $B$ in $\mathbb{H}^{\alpha,\ep}_m\setminus\mathcal{B}^{\alpha,\ep}_0$.
Then we have that $\mathcal{S}_{\alpha,\ep}(t)B\subseteq\mathcal{B}^{\alpha,\ep}_0$, in $\mathbb{H}^{\alpha,\ep}_m$, for all $t\ge t_0$, where
\begin{equation}  \label{time-t0}  
t_0 := \max\left\{ \frac{1}{\nu_3}\ln(E(0)),0 \right\}.
\end{equation}
This completes the proof.
\end{proof}

\begin{remark}  \label{r:uniform-radius}
According to the proof, $\nu_3$ is a function of $\delta_0$ and the relation is $\nu_3\sim 1-c\delta^2_0>0$ for a sufficiently small constant $c>0.$
\end{remark}

\begin{remark}
The following global uniform bound follows immediately from estimate \eqref{ap-1} and \eqref{ap-5}.
Under the assumptions of Lemma \ref{t:diss-1}, there holds
\begin{equation}  \label{unif-bnd}
\limsup_{t\rightarrow+\infty} \|\zeta(t)\|_{\mathbb{H}^{\alpha,\ep}_m} \le E(0) + \frac{1}{\nu_3}Q(m) =: Q(\|\zeta_0\|_{\mathbb{H}^{\alpha,\ep}_m},m)
\end{equation}
for a monotonically increasing function $Q$, independent of $\alpha$ and $\ep$
\end{remark}

\subsection{Global attractors for Problem P$_{\alpha,\ep}$}

The main result in this section is 

\begin{theorem}  \label{t:global}
For each $\alpha\in(0,1]$ and $\ep\in(0,1]$ the semigroup $S_{\alpha,\ep}=(S_{\alpha,\ep}(t))_{t\ge0}$ admits a global attractor $\mathcal{A}^{\alpha,\ep}$ in $\mathbb{H}^{\alpha,\ep}_m$. 
The global attractor is invariant under the semiflow $S_{\alpha,\ep}$ (both positively and negatively) and attracts all nonempty bounded subsets of $\mathbb{H}^{\alpha,\ep}_m$; precisely, 

\begin{description}

\item[1] for each $t\geq 0$, $S_{\alpha,\ep}(t)\mathcal{A}^{\alpha,\ep} = \mathcal{A}^{\alpha,\ep}$, and 

\item[2] for every nonempty bounded subset $B$ of $\mathbb{H}^{\alpha,\ep}_m$,
\[
\lim_{t\rightarrow\infty}{\rm{dist}}_{\mathbb{H}^{\alpha,\ep}_m}(S_{\alpha,\ep}(t)B,\mathcal{A}^{\alpha,\ep}) := \lim_{t\rightarrow\infty}\sup_{\zeta\in B}\inf_{\xi\in\mathcal{A}^{\alpha,\ep}}\|S_{\alpha,\ep}(t)\zeta-\xi\|_{\mathbb{H}^{\alpha,\ep}_m} = 0.
\]

\end{description}

\noindent \noindent Additionally,

\begin{description}

\item[3] the global attractor is unique maximal compact invariant subset in $\mathbb{H}^{\alpha,\ep}_m$ given by
\[
\mathcal{A}^{\alpha,\ep} := \omega (\mathcal{B}^{\alpha,\ep}_0) := \bigcap_{s\geq 0}{\overline{\bigcup_{t\geq s}S_{\alpha,\ep}(t)\mathcal{B}^{\alpha,\ep}_0}}^{\mathbb{H}^{\alpha,\ep}_m}.
\]

\end{description}

\noindent Furthermore, 

\begin{description}

\item[4] The global attractor $\mathcal{A}^{\alpha,\ep}$ is connected and given by the union of the unstable manifolds connecting the equilibria of $S_{\alpha,\ep}(t)$.

\item[5] For each $\zeta_0=(\phi_0,\theta_0)^{tr}\in\mathbb{H}^{\alpha,\ep}_m$, the set $\omega(\zeta_0)$ is a connected compact invariant set, consisting of the fixed points of $S_{\alpha,\ep}(t).$

\end{description}

\end{theorem}

With the existence of a bounded absorbing set set $\mathcal{B}^0_{\alpha,\ep}$ (in Lemma \ref{t:diss-1}), the existence of a global attractor now depends on the precompactness of the semigroup of solution operators $S_{\alpha,\ep}$. 
We begin by discussing the precompactness of the second component $\theta$ which follows from a straight forward result. 
Indeed, the next result refers to the instantaneous regularization of the ``thermal'' function $\theta$. 
This result will also be useful later in Section \ref{s:reg}.

\begin{lemma}  \label{t:thermal-reg}
Under the assumptions of Lemma \ref{t:diss-1}, the global weak solutions to Problem P$_{\alpha,\ep}$ satisfy the following: for every $\tau>0$, 
\begin{equation}  \label{tcp-0}
\theta\in L^\infty(\tau,\infty;V) \cap L^2(\tau,\infty;H^2(\Omega)),
\end{equation}
and, for all $t\ge\tau,$ there hold the bounds,
\begin{equation}  \label{tcp-0.1}
\|\theta(t)\|_V \le Q_{\alpha,\ep}(\|\zeta_0\|_{\mathbb{H}^{\alpha,\ep}_m},m)
\end{equation}
where $Q_{\alpha,\ep}\sim\{\alpha^{-1/2},\ep^{-1/2}\}$, and
\begin{equation}  \label{tcp-0.111}
\int_0^t \|\theta(s)\|^2_{H^2(\Omega)} ds \le Q_{\alpha}(\|\zeta_0\|_{\mathbb{H}^{\alpha,\ep}_m},m),
\end{equation}
where $Q_\alpha\sim\alpha^{-1}$.
\end{lemma}

\begin{proof}
The result follows from a standard density argument (cf. e.g. \cite[pp. 243-244]{Zheng04}).
We return to the beginning of the proof of Theorem \ref{t:existence} by letting $\theta_0\in D(A_N)=\{\psi\in H^2(\Omega):\partial_n\psi=0\}$, $\vartheta=-\Delta\theta_n$, and $T>0$.
In place of \eqref{wk-9.1}, we find there holds 
\begin{align}
\frac{d}{dt}\ep\|\nabla\theta_n\|^2 + \|\Delta\theta_n\|^2 \le \frac{\delta_0^2}{\alpha}\alpha\|\phi_n'\|^2.  \label{edc-0}
\end{align}
Multiplying \eqref{edc-0} by $t$ to then integrate over $(0,T)$ yields,
\begin{align}
t\ep\|\theta_n(t)\|^2_V + \int_0^t s\|\Delta\theta_n(s)\|^2 ds & \le \int_0^t \left( \frac{\delta_0^2}{\alpha}s\cdot \alpha\|\phi_n'(s)\|^2 + \|\theta_n(s)\|^2_V \right) ds  \notag \\ 
& \le \frac{\delta_0^2}{\alpha}t \int_0^t \alpha\|\phi_n'(s)\|^2 ds + \int_0^t \|\theta_n(s)\|^2_V ds.  \label{edc-0.2}
\end{align}
Here we integrate \eqref{ap-15} on $(0,T)$ after omitting the positive terms $\nu_3 E + \|\phi_t\|^2_{V'}$ from the left-hand side to find the bounds
\begin{align}
\frac{\delta_0^2}{\alpha} t \int_0^t \alpha\|\phi'_n(s)\|^2_V ds & \le \frac{\delta_0^2}{\alpha} t \int_0^t \alpha\|\phi_t(s)\|^2_V ds  \notag \\ 
& \le \frac{\delta^2_0}{\alpha} E(0)\cdot t + \frac{\delta^2_0}{\alpha} Q(m)\cdot t^2  \label{edc-0.3}
\end{align}
and
\begin{align}
\int_0^t \|\theta_n(s)\|^2_V ds & \le \int_0^t \|\theta(s)\|^2_V ds  \notag \\ 
& \le \frac{1}{\nu_3}E(0) + \frac{1}{\nu_3}Q(m)\cdot t.  \label{edc-0.4}
\end{align}
When we combine \eqref{edc-0.2}-\eqref{edc-0.4} and choose any $0<\tau<T$, we find, for all $\tau\le t <T$,
\begin{align}
\|\theta_n(t)\|^2_V \le \frac{1}{\ep}E(0)\left( \frac{\delta^2_0}{\alpha}+\frac{1}{\nu_3\tau} \right) + \frac{1}{\ep}Q(m)\left( \frac{\delta_0^2}{\alpha}T+\frac{1}{\nu_3} \right).  \label{edc-1}
\end{align}
Moreover, for every $\tau>0$ and $t\ge\tau$ such that $\tau\le t < T,$
\begin{align}
\int_0^t \|\Delta\theta_n(s)\|^2 ds & \le \frac{1}{\tau}E(0)\left( \frac{\delta^2_0}{\alpha}+\frac{1}{\nu_3\tau} \right) + \frac{1}{\tau}Q(m)\left( \frac{\delta_0^2}{\alpha}T+\frac{1}{\nu_3} \right).  \label{edc-2}
\end{align}
(Observe, these bounds are independent of $t$ and $n$.)
Thus, there is $\theta\in L^\infty(\tau,T;V) \cap L^2(\tau,T;D(A_N))$ such that up to a subsequence (not relabeled), as $n\rightarrow\infty$
\begin{eqnarray}
\theta_n \rightharpoonup \theta & \text{weakly-* in} & L^\infty(\tau,T;V),
\label{edc-3} \\
\theta_n \rightharpoonup \theta & \text{weakly in} & L^2(\tau,T;D(A_N)).
\notag
\end{eqnarray}
For the heat equation \eqref{rel-3}, the $H^2$-elliptic regularity estimate is
\begin{align}
\|\theta\|_{H^2(\Omega)} \le C\left( \|A_N\theta\| + \delta_0\|\phi_t\| \right),
\end{align}
thus, for the above bounds we also find 
\begin{eqnarray}
\theta_n \rightharpoonup \theta & \text{weakly in} & L^2(\tau,T;H^2(\Omega)).
\label{edc-4}
\end{eqnarray}

In order to recover the result for $\theta_0\in H$, recall that $D(A_N)$ is dense in $H$, so for any $\theta_0\in H$, there is a sequence $(\theta_{0n})_{n=1}^{\infty}\subset D(A_N)$ such that $\theta_{0n}\rightarrow\theta_0$ in $H$.
Therefore, for any $\theta_0\in H$ and $T>0$ we deduce \eqref{edc-0}-\eqref{edc-4} hold as well.
Finally, the required bound \eqref{tcp-0.1} follows from \eqref{edc-1}, and \eqref{tcp-0.111} follows from \eqref{edc-2}.
This completes the proof.
\end{proof}

The precompactness of the semigroup of solution operators $S_{\alpha,\ep}$ now depends on the precompactness of the first component.
To this end we will show there is a $t_*>0$ such that the map $S_{\alpha,\ep}(t_*)$ is a so-called $\alpha$-contraction on $\mathcal{B}_0$; that is, there is a time $t_*>0$, a constant $0<\nu<1$ and a precompact pseudometric $M_*$ on $\mathcal{B}_0$, where $\mathcal{B}_0$ is the bounded absorbing set from Lemma \ref{t:diss-1}, such that for all $\zeta_1,\zeta_2 \in \mathcal{B}_0$,
\begin{equation}
\|S_{\alpha,\ep}(t_*)\zeta_1 - S_{\alpha,\ep}(t_*)\zeta_2\|_{\mathcal{H}_0} \le \nu\|\zeta_1 - \zeta_2\|_{\mathcal{H}_0} + M_*(\zeta_1,\zeta_2). \label{pseudo-m}
\end{equation}
Such a contraction is commonly used in connection with phase-field type equations as an alternative to establish the precompactness of a semigroup; for some particular recent results see, \cite{Grasselli-2012,Grasselli-Schimperna-2011,Zheng&Milani05}.

\begin{lemma}  \label{t:pseudometric}
Under the assumptions of Proposition \ref{t:cont-dep} where $\zeta_{01},\zeta_{02}\in\mathcal{B}_0$, there is a positive constant $\bar\nu_4=\bar\nu_4(J,\Omega),$ such that for all $t\ge0,$ 
\begin{align}
\|\zeta_1(t)-\zeta_2(t)\|^2_{\mathbb{H}^{\alpha,\ep}_m} & \le e^{-\bar\nu_4t} \|\zeta_1(0)-\zeta_2(0)\|^2_{\mathbb{H}^{\alpha,\ep}_m} + C_1\int_0^t \|\phi_1(s)-\phi_2(s)\|^2 ds  \notag \\ 
& + C_2\left(1 + e^{\bar\nu_1 t}\right)\left( |M_1-M_2|^2 + |N_1-N_2|^2 \right) + e^{\bar\nu_1 t} \|\zeta_1(0)-\zeta_2(0)\|^2_{\mathbb{H}^{\alpha,\ep}_m},  \label{alphac-0}
\end{align}
where $C_1>0$ depends on $\delta_0,$ $c_J$, and the embedding $H\hookrightarrow V'$, $C_2>0$ depends on $F$, $J$, $\Omega,$ $\delta_0$, and $c_J$, and where the constant $\bar\nu_1$ is given in Proposition \ref{t:cont-dep}.
Consequently, there is $t_*>0$ such that the operator $S_{\alpha,\ep}(t_*)$ is a strict contraction up to the precompact pseudometric on $\mathcal{B}_0$, in the sense of \eqref{pseudo-m}, given by
\begin{align}
M_*(\zeta_{01},\zeta_{02}) & := 
C_* \left( \int_0^{t_*} \|\phi_1(s)-\phi_2(s)\|^2 ds + |M_1-M_2|^2 + |N_1-N_2|^2 + \|\zeta_{01}-\zeta_{02}\|^2_{\mathbb{H}^{\alpha,\ep}_m} \right)^{1/2},  \label{pseudo-2}
\end{align}
where $C_*>0$ depends on $t_*$ and $\bar\nu_1,$ but is independent of $t$, $\alpha$, and $\ep$.
Furthermore, $S_{\alpha,\ep}$ is precompact on $\mathcal{B}_0$.
\end{lemma}

\begin{proof}
The proof is based on the proof of Proposition \ref{t:cont-dep}.
Here we multiply \eqref{diff-1}-\eqref{diff-3} by, respectively, $A_N^{-1}(\bar\phi-\langle\bar\phi\rangle)$, $\bar\phi-\langle\bar\phi\rangle$ and $\bar\theta$, then sum the resulting identities to yield, 

\begin{align}
& \frac{d}{dt} \|\bar\zeta\|^2_{\mathbb{H}^{\alpha,\ep}_m} + 2\|A^{1/2}_N\bar\theta\|^2 + 2(a\bar\phi+F'(\phi_1)-F'(\phi_2),\bar\phi) - 2(J*\bar\phi,\bar\phi)  \notag \\ 
& = 2\delta(\bar\theta,\bar\phi) - 2\delta(\bar\phi_t,\bar\theta) + 2\langle \bar\phi \rangle \langle \bar\mu \rangle |\Omega|.  \label{alphac-1}
\end{align}
This time estimating the resulting products using assumption (H2) yields, 
\begin{align}
2(a\bar\phi+F'(\phi_1)-F'(\phi_2),\bar\phi) & \ge 2c_0\|\bar\phi\|^2  \notag \\ 
& \ge c_0C_\Omega^{-2}\|\bar\phi\|^2_{V'} + c_0\|\bar\phi\|^2,  \label{alphac-2}
\end{align}
where we recall the continuous embedding $H\hookrightarrow V'$.
We also write, 
\begin{align}
-(J*\phi,\phi) & \ge -\|J\|_{L^1(\Omega)}\|\bar\phi\|^2  \notag \\ 
& = -c_J\|\bar\phi\|^2,  \label{alphac-3}
\end{align}
\begin{align}
2\delta(\bar\theta,\bar\phi) & \le 2\delta_0 \|A^{1/2}_N\bar\theta\| \|A^{-1/2}_N\bar\phi\|  \notag \\ 
& \le \frac{1}{2}\|A^{1/2}_N\bar\theta\|^2 + 2\delta_0^2\|\bar\phi\|^2_{V'},  \label{alphac-4}
\end{align}
and, 
\begin{align}
-2\delta(\bar\phi_t,\bar\theta) & \le 2\delta_0 \|A^{-1/2}_N\bar\phi_t\| \|A^{1/2}_N\bar\theta\|  \notag \\ 
& \le 2\delta_0^2\|\bar\phi_t\|^2_{V'} + \frac{1}{2}\|A^{1/2}_N\bar\theta\|^2.  \label{alphac-4.5}
\end{align}
Combining \eqref{alphac-1}-\eqref{alphac-4.5}, then applying the Poincar\'{e} inequality \eqref{Poincare2}, we have, for almost all $t\in[0,T],$
\begin{align}
\frac{d}{dt} & \|\bar\zeta\|^2_{\mathbb{H}^{\alpha,\ep}_m} + c_0C_\Omega^{-2}\|\bar\phi\|^2_{V'} + c_0\cdot\alpha\|\bar\phi\|^2 + \|\bar\theta\|^2_V  \notag \\ 
& \le 2\delta_0^2\|\bar\phi\|^2_{V'} + c_J\|\bar\phi\|^2 + 2\delta_0^2\|\bar\phi_t\|^2_{V'} + 2|\langle \bar\phi \rangle| |\langle \bar\mu \rangle| |\Omega| + |\langle \bar\theta \rangle|.  \label{alphac-5}
\end{align}
We readily find that there is a positive constant (independent of $\alpha\in(0,1]$),
\[
\bar\nu_4=\bar\nu_4(J,\Omega):=\min\left\{ c_0C^{-2}_\Omega, c_0, c^{-1}_\Omega \right\},
\]
such that \eqref{alphac-5} becomes, with \eqref{diff-19.1}, for almost all $t\in[0,T],$ 
\begin{align}
& \frac{d}{dt} \|\bar\zeta\|^2_{\mathbb{H}^{\alpha,\ep}_m} + \bar\nu_4\|\bar\zeta\|^2_{\mathbb{H}^{\alpha,\ep}_m}  \notag \\ 
& \le 2\delta_0^2\|\bar\phi\|^2_{V'} + c_J\|\bar\phi\|^2 + 2\delta_0^2\|\bar\phi_t\|^2_{V'} + 2|\langle \bar\phi \rangle| \mu_* |\Omega| + |\langle \bar\theta \rangle|.  \label{alphac-6}
\end{align}
After applying Gr\"{o}nwall's inequality to \eqref{alphac-6}, we obtain, for all $t\ge0,$ 
\begin{align}
\|\bar\zeta(t)\|^2_{\mathbb{H}^{\alpha,\ep}_m} & \le e^{-\bar\nu_4t}\|\bar\zeta(0)\|^2_{\mathbb{H}^{\alpha,\ep}_m} + \int_0^t \left( 2\delta_0^2\|\bar\phi(s)\|^2_{V'} + c_J\|\bar\phi(s)\|^2 + 2\delta_0^2\|\bar\phi_t(s)\|^2_{V'} \right) ds  \notag \\ 
& + \frac{1}{\bar\nu_4} \left( 2|\langle \bar\phi \rangle| \mu_* |\Omega| + |\langle \bar\theta \rangle| \right).  \label{alphac-7}
\end{align}
It is important to note that by \eqref{diff-0},
\begin{align}
\int_0^t 2\delta_0^2\|\bar\phi_t(s)\|^2_{V'} ds & \le 2\delta_0^2e^{\bar\nu_1 t} \left( \|\bar\zeta(0)\|^2_{\mathbb{H}^{\alpha,\ep}_m} + \frac{2\bar\nu_2}{\bar\nu_1} \left( |\langle \bar\phi \rangle|+|\langle \bar\theta \rangle| \right)^2  \right)  \notag \\ 
& \le C e^{\bar\nu_1 t} \left( \|\bar\zeta(0)\|^2_{\mathbb{H}^{\alpha,\ep}_m} + \left( |\langle \bar\phi \rangle|+|\langle \bar\theta \rangle| \right)^2  \right),  \label{alphac-8}
\end{align}
where $C=C(F,J,\Omega,\delta_0)>0$.
Moreover, with \eqref{diff-19.1} again, 
\begin{align}
\frac{1}{\bar\nu_4} \left( 2|\langle \bar\phi \rangle| \mu_* |\Omega| + |\langle \bar\theta \rangle| \right) & \le C\left( |\langle \bar\phi \rangle|^2 + |\langle \bar\theta \rangle|^2 \right),  \label{alphac-9}
\end{align}
where here $C>0$ depends on $c_J$, $\delta_0$, and the Lipschitz bound on $F'$.
Together \eqref{alphac-7}-\eqref{alphac-9} yield the estimate \eqref{alphac-0}.

Clearly there is a $t_*>0$ so that $e^{-\bar\nu_4 t_*/2}<1.$
Thus, the operator $S_{\alpha,\ep}(t_*)$ is a strict contraction up to the pseudometric $M_*$ defined by \eqref{pseudo-2}.
The pseudometric $M_*$ is precompact thanks to the Aubin-Lions compact embedding (cf. e.g. \cite[Theorem 3.1.1]{Zheng04}):
\[
\left\{ \chi\in L^2(0,t_*;V) : \chi_t \in L^2(0,t_*;V') \right\} \hookrightarrow L^2(0,t_*;H).
\]
Finally, with the compactness result for the second component given in Lemma \ref{t:thermal-reg}, the operators $S_{\alpha,\ep}$ are precompact on $\mathbb{H}^{\alpha,\ep}_m$.
The proof is complete.
\end{proof}

\begin{proof}[Proof of Theorem \ref{t:global}]
The precompactness of the solution operators $S_{\alpha,\ep}$ follows via the method of precompact pseudometrics (see Lemma \ref{t:thermal-reg} and Lemma \ref{t:pseudometric}).
With the existence of a bounded absorbing set $\mathcal{B}^{\alpha,\ep}_0$ in $\mathbb{H}^{\alpha,\ep}_m$ (Lemma \ref{t:diss-1}), the existence of a global attractor in $\mathbb{H}^{\alpha,\ep}_m$ is well-known and can be found in \cite{Temam88,Babin&Vishik92} for example.
Additional characteristics of the attractor follow thanks to the gradient structure of Problem P$_{\alpha,\ep}$ (Remark \ref{r:gradient}).
In particular, the first three claims in the statement of Theorem \ref{t:global} are a direct result of the existence of the an absorbing set, a Lyapunov functional $\mathcal{E}_\ep$, and the fact that the system $(\mathcal{X}^{\alpha,\ep}_m,S_{\alpha,\ep}(t),\mathcal{E}_\ep)$ is gradient. 
The fourth property is a direct result \cite[Theorem VII.4.1]{Temam88}, and the fifth follows from \cite[Theorem 6.3.2]{Zheng04}.
This concludes the proof.
\end{proof}

\subsection{Further uniform estimates and regularity properties for Problem P$_{\alpha,\ep}$}  \label{s:reg}

Our next aim is to bound the global attractor in a more regular space by showing the existence of an absorbing set in $\mathbb{V}^{\alpha,\ep}_m$.
Once this is established, we will bound the ($\alpha$-weighted) chemical potential $\sqrt{\alpha}\mu$ in $H^2(\Omega),$ which also establishes a bound in $L^\infty(\Omega).$
Some of the results in this subsection require hypothesis (H5) with $q\ge2$, and hence the existence of a global attractor for Problem P$_{\alpha,\ep}.$

\begin{lemma}
Under the assumptions of Lemma \ref{t:diss-1}, the set given by
\begin{equation}  \label{reg-set}
\mathcal{B}^{\alpha,\ep}_1 := \left\{ \zeta\in\mathbb{V}^{\alpha,\ep}_m : \|\zeta\|^2_{\mathbb{V}^{\alpha,\ep}_m} \le \left( \frac{1}{\ep}+1 \right) \left( E(0) + \left( \frac{2}{\nu_3}+1 \right) Q_\alpha(m) +1 \right) \right\},
\end{equation}
for some positive monotonically increasing function $Q_\alpha\sim\alpha^{-1}$, is a closed, bounded absorbing set in $\mathbb{V}^{\alpha,\ep}_m$, positively invariant under the semigroup $S_{\alpha,\ep}$.
\end{lemma}

\begin{proof}
Because we already know the existence of an absorbing set in $\mathbb{H}^{\alpha,\ep}_m$, bounded uniformly in $\alpha$ and $\ep$, the proof is relatively simple and follows a very standard idea (cf. e.g. \cite[Section 11.1.2]{Robinson94}).
Multiply \eqref{rel-1}-\eqref{rel-3} $\phi,$ $A_N\phi,$ and $A_N\theta$, respectively, then sum the resulting identities to find, 
\begin{align}
\frac{1}{2}\frac{d}{dt} & \{ \|\phi\|^2 + \alpha\|\nabla\phi\|^2 + \ep\|\nabla\theta\|^2 \} + ((\nabla a)\phi + a\nabla\phi - \nabla J*\phi + F''(\phi)\nabla\phi, \nabla\phi) + \|\Delta\theta\|^2  \notag \\
& = \delta(\nabla\theta,\nabla\phi) + \delta(\phi_t,\Delta\theta).  \label{buk-1}
\end{align}
Recalling the scheme supporting \eqref{wk-21}, we have 
\begin{align}
((\nabla a)\phi + a\nabla\phi - \nabla J*\phi + F''(\phi)\nabla\phi,\nabla\phi) & \ge \frac{c_0}{2}\|\nabla\phi\|^2 - \frac{1}{c_0}\left( \|J_k\|_{W^{1,\infty}(\Omega)}+d_J^2 \right)\|\phi\|^2.  \label{buk-2}
\end{align}
We estimate the remaining two products on the right-hand side of \eqref{buk-1} as,
\begin{align}
\delta(\nabla\theta,\nabla\phi) \le \frac{\delta^2_0}{c_0}\|\nabla\theta\|^2 + \frac{c_0}{4}\|\nabla\phi\|^2,  \label{buk-3}
\end{align}
and
\begin{align}
\delta(\phi_t,\Delta\theta) \le \delta^2_0\|\phi_t\|^2 + \|\Delta\theta\|^2.  \label{buk-4}
\end{align}
Together, \eqref{buk-1}-\eqref{buk-4} produce,
\begin{align}
\frac{d}{dt} & \{ \|\phi\|^2 + \alpha\|\nabla\phi\|^2 + \ep\|\nabla\theta\|^2 \} + \frac{c_0}{2}\|\nabla\phi\|^2  \notag \\ 
& \le \frac{2}{c_0}\left( \|J_k\|_{W^{1,\infty}(\Omega)}+d_J^2 \right)\|\phi\|^2 + \frac{2\delta_0^2}{c_0}\|\nabla\theta\|^2 + 2\delta^2_0\|\phi_t\|^2.  \label{buk-5}
\end{align}
For $t\ge1$, integrating \eqref{buk-5} over $t-1<s<t$ yields,
\begin{align}
& \|\phi(t)\|^2 + \alpha\|\nabla\phi(t)\|^2 + \ep\|\nabla\theta(t)\|^2 + \frac{c_0}{2}\int_s^t \|\nabla\phi(\sigma)\|^2 d\sigma  \notag \\ 
& \le \|\phi(s)\|^2 + \alpha\|\nabla\phi(s)\|^2 + \ep\|\nabla\theta(s)\|^2  \notag \\
& + \frac{2}{c_0}\left( \|J_k\|_{W^{1,\infty}(\Omega)}+d_J^2 \right) \int_s^t \|\phi(\sigma)\|^2 d\sigma + \frac{2\delta_0^2}{c_0}\int_s^t \|\nabla\theta(\sigma)\|^2 d\sigma + \frac{2\delta^2_0}{\alpha}\int_s^t \alpha\|\phi_t(\sigma)\|^2 d\sigma.  \notag
\end{align}
Hence, using the bounds \eqref{ap-18} and \eqref{unif-bnd} (also see \eqref{ap-5}), we find
\begin{align}
& \|\phi(t)\|^2 + \alpha\|\nabla\phi(t)\|^2 + \ep\|\nabla\theta(t)\|^2  \notag \\ 
& \le \|\phi(s)\|^2 + \alpha\|\nabla\phi(s)\|^2 + \ep\|\nabla\theta(s)\|^2 + (e^{-\nu_3t}+1)E(0) + \left( \frac{2}{\nu_3}+1 \right) Q_\alpha(m),  \label{buk-6}
\end{align}
where $\nu_3>0$ is described in Lemma \ref{t:diss-1} and $Q_\alpha\sim\alpha^{-1}.$
Then integrating \eqref{buk-6} with respect to $s$ on $(t-1,t)$ shows,
\begin{align}
& \|\phi(t)\|^2 + \alpha\|\nabla\phi(t)\|^2 + \ep\|\nabla\theta(t)\|^2  \notag \\ 
& \le \frac{1}{\ep}\int_{t-1}^t \left( \|\phi(s)\|^2 + \alpha\|\nabla\phi(s)\|^2 + \ep\|\nabla\theta(s)\|^2 \right) ds + (e^{-\nu_3t}+1)E(0) + \left( \frac{2}{\nu_3}+1 \right) Q_\alpha(m).  \notag
\end{align}
Once again we rely on \eqref{ap-18} (hence the factor of $\ep^{-1}$ above) to find
\begin{align}
\|\phi(t)\|^2 + \alpha\|\nabla\phi(t)\|^2 + \ep\|\nabla\theta(t)\|^2 \le \left( \frac{1}{\ep}+1 \right) \left( (e^{-\nu_3t}+1)E(0) + \left( \frac{2}{\nu_3}+1 \right) Q_\alpha(m) \right).  \label{buk-8}
\end{align}
Hence, the left-hand side does eventually go into a ball. 
With estimate \eqref{buk-8}, we deduce the existence of the regular absorbing set $\mathcal{B}^{\alpha,\ep}_1$.
This completes the proof.
\end{proof}

\begin{remark}
We draw two useful facts from \eqref{buk-8}.
The first is the time uniform bound
\begin{align}
\limsup_{t\rightarrow+\infty}\|\zeta(t)\|^2_{\mathbb{V}^{\alpha,\ep}_m} & \le \left( \frac{1}{\ep}+1 \right) \left( E(0) + \frac{1}{\nu_3}Q_{\alpha}(m) \right)  \notag \\
& =:Q_{\alpha,\ep}(\|\zeta_0\|_{\mathbb{H}^{\alpha,\ep}_m},m).  \label{buk-10}
\end{align}
This bound becomes arbitrarily large as $\alpha\rightarrow0^+$ or $\ep\rightarrow0^+$.
Second, the ``time of entry'' of any nonempty bounded subset $B$ of $\mathbb{V}^{\alpha,\ep}_m$ in $\mathcal{B}^{\alpha,\ep}_1$ under the solution operator $S_{\alpha,\ep}(t)$ is given by
\begin{align}
t_1 := \max \left\{ \frac{1}{\nu_3} \ln E(0),0 \right\}.  \notag
\end{align}
\end{remark}

The following result now follows in a standard way (cf. e.g. \cite{Temam88}).

\begin{corollary}
For each $\alpha\in(0,1]$ and $\ep\in(0,1]$, the global attractor $\mathcal{A}^{\alpha,\ep}$ is bounded in $\mathbb{V}^{\alpha,\ep}_m$, i.e., $\mathcal{A}^{\alpha,\ep}\subset\mathcal{B}^{\alpha,\ep}_1,$ and compact in $\mathbb{H}^{\alpha,\ep}_m$.
\end{corollary}

\begin{remark}
The radius of the absorbing set $\mathcal{B}^{\alpha,\ep}_1$ in $\mathbb{H}^{\alpha,\ep}_m$ may be larger than the radius of $\mathcal{B}^{\alpha,\ep}_1$ in $\mathbb{V}^{\alpha,\ep}_m$.
This is due to the (compact) embedding $\mathbb{V}^{\alpha,\ep}_m\hookrightarrow\mathbb{H}^{\alpha,\ep}_m$.
Moreover, from \eqref{buk-8} we find that the ``radius'' of the set $\mathcal{B}^{\alpha,\ep}_1$ depends on $\alpha$ and $\ep$ like, respectively, $\alpha^{-1}$ and $\ep^{-1}$.
\end{remark}

The following result refers to the instantaneous regularization of the $\alpha$-weighted chemical potential $\sqrt{\alpha}\mu$. 

\begin{lemma}
Under the assumptions of Lemma \ref{t:diss-1}, the global weak solutions to Problem P$_{\alpha,\ep}$ satisfy the following, for every $\tau>0$,
\begin{align}
\sqrt{\alpha}\mu\in L^\infty(\tau,\infty;D(A_N)) \quad \text{and} \quad \mu\in L^\infty(\tau,\infty;V),  \label{mu-1}
\end{align}
and for all $t\ge\tau$ there holds,
\begin{equation}  \label{mu-2}
\alpha\|\Delta\mu(t)\|^2 + \|\mu(t)\|^2_V \le \frac{1}{\alpha} \left( Q(\|\zeta_0\|_{\mathbb{H}^{\alpha,\ep}_m}) + \frac{1}{\nu_3}Q(m) \right),
\end{equation}
where $\nu_3$ and $Q$ are due to \eqref{ap-1} (hence, the right-hand side of \eqref{mu-2} is independent of $\ep$, but dependent on $\delta_0$ like $\frac{1}{\nu_3}\sim\delta_0^{-1}$).
\end{lemma}

\begin{proof}
To begin, multiply \eqref{rel-1} and \eqref{rel-2} by, respectively, $\alpha A_N\mu$ and $A_N\mu$ in $L^2(\Omega)$ and sum the resulting identities.
(Recall that with $\zeta_0=(\phi_0,\theta_0)^{tr}\in\mathbb{H}^{\alpha,\ep}_m$, we only know that $\mu\in L^2(0,T;V)$ by \eqref{ws-3}; hence, $\Delta\mu\in L^2(0,T;V')$). 
Hence, the aforementioned multiplication is formal, but can be rigorously justified using the above Galerkin approximation procedure).
We then have 
\begin{align}
\alpha\|A_N\mu\|^2 + \|\nabla\mu\|^2 & = (a\phi-J*\phi+F'(\phi)-\delta\theta,A_N\mu).  \notag
\end{align}
After applying a basic estimate to the right-hand side, we easily arrive at
\begin{align}
\alpha\|A_N\mu\|^2 + \|\nabla\mu\|^2 & \le \frac{4}{\alpha}\left( \|\sqrt{a}\phi\|^2 + c^2_J\|\phi\|^2 + \|F'(\phi)\|^2 + \delta^2_0\|\theta\|^2 \right),  \notag
\end{align}
to which we employ the bounds \eqref{Fcons-3.1}, \eqref{ap-1} and \eqref{mean-mu} to find \eqref{mu-2} as claimed. 
\end{proof}

\begin{lemma} 
Under the assumptions of Lemma \ref{t:diss-1}, the global weak solutions to Problem P$_{\alpha,\ep}$ satisfy the following, 
\begin{eqnarray}
\phi_t & \in & L^\infty(0,\infty;V'),  \label{new-1} \\ 
\sqrt{\alpha}\phi_t & \in & L^\infty(0,\infty;H),  \label{new-2} \\ 
\theta & \in & L^\infty(0,\infty;V).  \label{new-3} 
\end{eqnarray}
and there is a positive constant $\nu_5=\nu_5(\alpha,F)\sim\alpha^{-1}$, independent of $\zeta_0$, such that, for all $t\ge0$, there holds, 
\begin{align}
\|\phi_t(t)\|^2_{V'} + \alpha\|\phi_t(t)\|^2 + \|\theta(t)\|^2_V \le \left( e^{-\nu_5t} + \frac{1}{\nu_5} \right) Q(\|\zeta_0\|_{\mathbb{H}^{\alpha,\ep}_m},m),  \label{pro-0}
\end{align}
for some monotonically increasing function $Q$. (Observe, the right-hand side of \eqref{pro-0} can be bounded independent of $\alpha$.)
\end{lemma}

\begin{proof}
We will only give a formal derivation of \eqref{pro-0} as the remaining details are justified within the Galerkin approximation scheme already given in the beginning of this section. 
Now, we differentiate \eqref{rel-1} and \eqref{rel-2} with respect to $t$ and write the resulting equations in the terms
\[
u:=\phi_t, \quad \varpi:=\theta_t, \quad m:=\mu_t,
\]
which now gives,
\begin{eqnarray}
u_t = \Delta m &\text{in}& \Omega\times(0,\infty)  \label{pro-1} \\ 
m = au - J*u + F''(\phi)u + \alpha u_t - \delta\varpi &\text{in}& \Omega\times(0,\infty)  \label{pro-2} \\ 
\ep\theta_t - \Delta\theta = -\delta\phi_t &\text{in}& \Omega\times(0,\infty)  \label{pro-3} \\ 
\partial_n m = 0 &\text{on}& \Gamma\times(0,\infty)  \label{pro-4} \\ 
\partial_n\theta = 0 &\text{on}& \Gamma\times(0,\infty)  \label{pro-5} \\ 
\alpha u(x,0) = \mu(0) - a\phi(0) + J*\phi(0) - F'(\phi(0)) + \delta\theta(0) &\text{at}& \Omega\times\{0\}  \label{pro-6} \\ 
\theta(x,0) = \theta_0 &\text{at}& \Omega\times\{0\}. \label{pro-7}  
\end{eqnarray}
Multiply \eqref{pro-3} by $\varpi$ in $L^2(\Omega)$, so
\begin{align}
\frac{d}{dt}\|\theta\|^2_V+2\ep\|\varpi\|^2=-2\delta(u,\varpi).  \label{pro-8}
\end{align}
Now, in $L^2(\Omega)$, multiply \eqref{pro-1} and \eqref{pro-2} by $A_N^{-1}u$ and $u$, respectively, and sum the resulting identity to \eqref{pro-8} to obtain (recall $\langle u \rangle=0$, so $A^{-1}_Nu$ belongs to $V_0'$),
\begin{align}
\frac{d}{dt} & \left\{ \|u\|^2_{V'}+\alpha\|u\|^2+\|\theta\|^2_V \right\} + 2((a+F''(\phi))u,u) + 2\ep\|\varpi\|^2 = 2(J*u,u). \label{pro-9}
\end{align}
Estimating the products in a similar fashion as we have already done above shows,
\begin{align}
2((a+F''(\phi))u,u) & \ge 2c_0\|u\|^2  \notag \\ 
& \ge c_0\|u\|^2 + c_0 C^{-2}_\Omega\|u\|^2_{V'},  \label{pro-11}
\end{align}
since $\|u\|_{V'}\le C_\Omega\|u\|$, and
\begin{align}
2(J*u,u) & \le 2c_J\|u\|^2.  \label{pro-12}
\end{align}
Combining \eqref{pro-9}-\eqref{pro-12},
\begin{align}
\frac{d}{dt} \left\{ \|u\|^2_{V'}+\alpha\|u\|^2+\|\theta\|^2_V \right\} + \nu_5\left( \|u\|^2_{V'} + \alpha\|u\|^2 + \|\theta\|^2_V \right) + 2\ep\|\varpi\|^2 \le 2c_J\|u\|^2,  \label{pro-13}
\end{align}
where $0<\nu_5=\nu_5(\alpha):=\min\{c_0 C^{-2}_\Omega,\frac{c_0}{\alpha},1\}\sim\alpha^{-1}.$
Since $u=\phi_t$ is uniformly bounded in $L^2(\Omega)$ (see \eqref{unif-bnd}), then integrating \eqref{pro-13} on $(0,t)$ produces,
\begin{align}
& \|u(t)\|^2_{V'}+\alpha\|u(t)\|^2+\|\theta(t)\|^2_V + \int_0^t \ep\|\varpi(s)\|^2 ds  \notag \\ 
& \le e^{-\nu_5t}\left( \|u(0)\|^2_{V'}+\alpha\|u(0)\|^2+\|\theta(0)\|^2_V \right) + Q(\|\zeta_0\|_{\mathbb{H}^{\alpha,\ep}_m},m).  \label{pro-14}
\end{align}
Observe $\varpi=\theta_t\in L^2(0,T;L^2(\Omega))$ (see \eqref{wk-0.3}) so we are free to omit the term.
Recall that the initial conditions are taken in the weak/$L^2$-sense, for all $\varphi\in V,$
\[
(\mu(0),\varphi) = \lim_{t\rightarrow0^+}(\mu(t),\varphi),
\]
hence, by \eqref{ws-3}, we conclude 
\begin{equation}  \label{pro-15}
\mu(0)\in V\hookrightarrow H.
\end{equation}
Similarly, with \eqref{ws-6} and \eqref{wk-44}, 
\begin{equation}  \label{pro-15.5}
\rho(\cdot,\phi(0)):=a\phi(0)+F'(\phi(0))\in V\hookrightarrow H.
\end{equation}
Then using \eqref{pro-6}, \eqref{continuity-1}, \eqref{pro-15}, and \eqref{pro-15.5}, 
\begin{equation}  \label{pro-16}
\alpha u(0) = \mu(0) - a\phi(0) + J*\phi(0) - F'(\phi(0)) + \delta\theta(0) \in H.
\end{equation}
It should also be mentioned that $\theta_0\in H$, while, with \eqref{L2-initial-theta}, for all $\vartheta\in V,$
\[
(\theta_0,\vartheta)=(\theta(0),\vartheta).
\]
Hence, the bound on the right-hand side of \eqref{pro-14} is well defined. 
This establishes \eqref{pro-0}.
This finishes the proof.
\end{proof}

The final result is this section concerns bounding the global attractor $\mathcal{A}^{\alpha,\ep}$ in a more regular subspace of $\mathbb{V}^{\alpha,\ep}_m$.
For each $m\ge0$, $\alpha\in(0,1]$ and $\ep\in(0,1]$, we now define the regularized phase-space
\begin{align}
\mathbb{W}^{\alpha,\ep}_m := \{ \zeta=(\phi,\theta)^{tr} \in \mathbb{V}^{\alpha,\ep}_m : \sqrt{\alpha}\mu\in H^2(\Omega),\ |\langle \phi \rangle|, |\langle \theta \rangle| \le m \},  \notag 
\end{align}
with the norm inherited from $\mathbb{V}^{\alpha,\ep}_m$.
Also, we define the following metric space
\begin{align}
\mathcal{Y}^{\alpha,\ep}_m := \left\{ \zeta=(\phi,\theta)^{tr}\in\mathbb{W}^{\alpha,\ep}_m : F(\phi)\in L^1(\Omega) \right\},  \notag
\end{align}
endowed with the metric 
\begin{align}
d_{\mathcal{Y}^{\alpha,\ep}_m}(\zeta_1,\zeta_2) := \|\zeta_1-\zeta_2\|_{\mathbb{V}^{\alpha,\ep}_m} + \left| \int_\Omega F(\phi_1)dx - \int_\Omega F(\phi_2)dx \right|^{1/2}.  \notag
\end{align}

\begin{theorem}  \label{t:global-attractor}
For each $\alpha\in(0,1]$, $\ep\in(0,1]$ and for any $t\ge t_*$, the semigroup $S_{\alpha,\ep}$ satisfies $S_{\alpha,\ep}(t):\mathcal{X}^{\alpha,\ep}_m\rightarrow\mathcal{Y}^\alpha_m.$
Moreover, the global attractor $\mathcal{A}^{\alpha,\ep}$ admitted by the semigroup $S_{\alpha,\ep}$ is bounded in $\mathbb{W}^{\alpha,\ep}_m$ and compact in $\mathbb{H}^{\alpha,\ep}_m.$
\end{theorem}

\begin{proof}
To begin, we let $\zeta_0=(\phi_0,\theta_0)^{tr}\in\mathbb{H}^{\alpha,\ep}_m$ be such that $F(\phi_0)\in L^1(\Omega)$ (i.e. $\zeta_0\in\mathcal{X}^{\alpha,\ep}_m$).
By the precompactness of the solution operators (see Remark \ref{r:gradient}), we know that, for all $t\ge t_*$, $S_{\alpha,\ep}(t)\zeta_0\in\mathbb{V}^{\alpha,\ep}_m$ ($t_*$ was given in Lemma \ref{t:pseudometric} and we may choose $\tau=t_*$ in Lemma \ref{t:thermal-reg}).
Letting $\zeta_1=S_{\alpha,\ep}(t)\zeta_0,$ it now suffices to show that $S_{\alpha,\ep}(t)\zeta_1\in\mathbb{W}^{\alpha,\ep}_m$ for all $t\ge t_*;$ i.e., we will show that 
\begin{align}
\sqrt{\alpha}\mu\in L^\infty(t_*,\infty;H^2(\Omega)).  \label{rob-3}
\end{align} 

Since $\phi_t=\Delta\mu,$ the estimate \eqref{pro-0} shows, 
\begin{align}
\|\nabla\mu(t)\|^2 + \alpha\|\Delta \mu(t)\|^2 \le \left( 1+\frac{1}{\nu_5} \right)Q(\|\zeta_0\|_{\mathbb{H}^{\alpha,\ep}_m},m).  \label{rop-1}
\end{align}
Adding $\langle\mu(t)\rangle^2$ to both sides of \eqref{rop-1} and applying the Poincar\'{e} inequality \eqref{Poincare2} (on the left) and \eqref{mu-2} (on the right), we now have
\begin{align}
c^{-1/2}_\Omega\|\mu(t)\|^2 + \alpha\|\Delta \mu(t)\|^2 \le Q(\|\zeta_0\|_{\mathbb{H}^{\alpha,\ep}_m},m)  \label{rop-2}
\end{align}
for some positive monotonically increasing function $Q.$
In this setting, the (standard) $H^2$-elliptic regularity estimate is
\begin{align}
\sqrt{\alpha}\|\mu\|_{H^2(\Omega)} & \le C(\sqrt{\alpha}\|A_N\mu\|+\|\mu\|)  \notag
\end{align}
for some positive constant $C,$ so with \eqref{rop-2} we readily find 
\begin{align}
\sqrt{\alpha}\|\mu\|_{H^2(\Omega)} & \le Q(\|\zeta_0\|_{\mathbb{H}^{\alpha,\ep}_m},m).
\end{align}
This establishes \eqref{rob-3} and completes the proof.
\end{proof}

\section{Conclusions and further remarks}

In this article we have shown that the relaxation Problem P$_{\alpha,\ep}$ is globally well-posed and generates a dissipative and conservative semigroup of solution operators which, in turn, admit a family of global attractors that possess a certain degree of regularity.
The relaxation problem considered here presented many difficulties due to the presence of the nonlocal diffusion terms on the order parameter $\phi$. 

Some interesting future work would include determining whether the (fractal) dimension of the global attractors found here is finite and {\em{independent}} of $\alpha$ and $\ep$.
Additionally, it would also be interesting to establish an upper-semicontinuity result for the family of global attractors when $\alpha\rightarrow0^+$ and $\ep\rightarrow0^+$ (compare this to the standard diffusion case in \cite{Gal&Miranville09}).

Hence, we should also examine the existence of an exponential attractor for Problem P$_{\alpha,\ep}$, and naturally, its basin of attraction.
With that result, we could seek a robustness result for the family of exponential attractors.
Examining problems related to stability (and hence the approximation of the longterm behavior of a relaxation problem to the associated limit problem) may prove to be an important source of further work on nonlocal Cahn-Hilliard and nonlocal phase field models.

Of course, some future work may examine several variants to the current model. 
Such variants may include a convection term that accounts for the effects of an averaged (fluid) velocity field, which naturally couples with a nonisothermal Navier-Stokes equation (on the former, see for example \cite{Porta-Grasselli-2014}).
Indeed, one may include nonconstant mobility in the nonlocal Cahn-Hilliard equation (cf. e.g. \cite{Frigeri-Grasselli-Rocca_2014}). 
It may be interesting to generalize the coupled heat equation to a Coleman-Gurtin type equation. 
Also, one may examine the associated nonlocal phase-field model \eqref{mot-1}, and the effects of generalizing the heat equation along the lines of \cite{GPS07,Herrera&Pavon02,Joseph-Preziosi-89,Joseph-Preziosi-90} where Fourier's law is replaced with a Maxwell-Cattaneo law because in this more realistic setting, ``disturbances'' propagate at a {\em{finite}} speed. 

It would also be interesting to study the nonlocal variant of the Cahn-Hilliard and phase-field equations by introducing relevant dynamic boundary conditions (again, see \cite{Gal&Miranville09}).
In this case, several interesting difficulties may arise concerning the regularity of solutions because, typically in applications, $H^1(\Omega)$ regularity (or better) is sought in order to define the trace of the solution; recall, $trace:H^s(\Omega)\rightarrow H^{s-1/2}(\Gamma).$ 
Additionally, we should study the case when the potential is singular (see hypotheses in \cite[Section 3]{Gal&Grasselli14}, for example).

\section*{Acknowledgments}

The author would like to thank Professor Ciprian G. Gal for recommending portions of this project. 
In addition, the author is indebted to the anonymous referee(s) for their careful reading of the manuscript, which undoubtably improved the paper.

\bigskip

\bibliographystyle{amsplain}
\providecommand{\bysame}{\leavevmode\hbox to3em{\hrulefill}\thinspace}
\providecommand{\MR}{\relax\ifhmode\unskip\space\fi MR }
\providecommand{\MRhref}[2]{%
  \href{http://www.ams.org/mathscinet-getitem?mr=#1}{#2}
}
\providecommand{\href}[2]{#2}

\end{document}